\documentclass{article}

\usepackage{PRIMEarxiv}

\usepackage[utf8]{inputenc} 
\usepackage[T1]{fontenc}    
\usepackage{hyperref}       
\usepackage{url}            
\usepackage{booktabs}       
\usepackage{amsfonts}       
\usepackage{nicefrac}       
\usepackage{microtype}      
\usepackage{lipsum}
\usepackage{fancyhdr}       
\usepackage{graphicx}       
\graphicspath{{media/}}     
\usepackage{comment}
\usepackage{amsmath}
\usepackage{amssymb}
\usepackage{mathtools}
\usepackage{amsthm}

\usepackage[utf8]{inputenc} 
\usepackage[T1]{fontenc}    
\usepackage{hyperref}       
\usepackage{url}            
\usepackage{booktabs}       
\usepackage{amsfonts}       
\usepackage{nicefrac}       
\usepackage{microtype}      

\usepackage{subfigure}
\usepackage{graphicx}
\usepackage[dvipsnames]{xcolor}

\usepackage{times}
\usepackage{epsfig}
\usepackage{amsmath}
\usepackage{epstopdf}
\usepackage{extarrows,xfrac,bm,nameref,dsfont,color,cancel}

\usepackage{enumitem}

\newcommand{\mL}{\mathcal{L}}
\newcommand{\WW}{\mathcal{W}}
\newcommand{\Wkp}{\mathcal{W}^{k,p}}

\newcommand{\Wsp}{\mathcal{W}^{s,p}}
\newcommand{\Wnp}{\mathcal{W}^{n,p}}

\newcommand{\Wni}{\mathcal{W}^{n,\infty}}
\newcommand{\X}{\mathcal{X}}
\newcommand{\Y}{\mathcal{Y}}
\newcommand{\pdf}{{\varrho}}

\let\emptyset\varnothing

\newcommand{\N}{\ifmmode\mathbb{N}\else$\mathbb{N}$\fi}
\newcommand{\Q}{\ifmmode\mathbb{Q}\else$\mathbb{Q}$\fi}
\newcommand{\Z}{\ifmmode\mathbb{Z}\else$\mathbb{Z}$\fi}

\newcommand{\tn}[1]{\textnormal{#1}}
\newcommand{\myto}{\mathop{\raisebox{-0pt}{\scalebox{2.6}[1]{$\longrightarrow$}}}}

\newcommand{\NN}{\tn{NN}}
\newcommand{\NNinput}{\tn{\#input}}
\newcommand{\NNwidth}{\, \tn{widthvec}}

\newcommand{\xal}{{\bm{\alpha}}}
\newcommand{\xx}{{\bm{x}}}
\newcommand{\yy}{{\bm{y}}}

\newcommand{\xk}{{\bm{k}}}
\newcommand{\OO}{{\mathcal{O}}}

\newcommand{\pkt}{{p_{f,\xk,t,s}}}
\newcommand{\pk}{{p_{\xk}}}

\newcommand{\I}{\mathcal{I}}
\newcommand{\J}{\mathcal{J}}

\newcommand{\zod}{{{(0,1)^d}}}

\newcommand{\tf}{{\tilde{f}}}

\newcommand{\OkK}{{\Omega_{\xk,K}}}
\newcommand{\BkK}{{B_{\xk,K}}}

\newcommand{\xth}{{\bm{\theta}}}
\renewcommand{\epsilon}{\varepsilon}
\renewcommand{\subset}{\subseteq}
\renewcommand{\frac}[2]{\tfrac{#1}{#2}}
\newcommand{\hz}[1]{\textcolor{blue}{\textbf{HZ: #1}}}

\makeatletter
\newcommand{\printfnsymbol}[1]{%
  \textsuperscript{\@fnsymbol{#1}}%
}

\usepackage{amsthm}
\theoremstyle{plain}
\newtheorem{theorem}{Theorem}[section]%
\newtheorem{corollary}[theorem]{Corollary} %
\newtheorem{lemma}[theorem]{Lemma}

\theoremstyle{definition}
\newtheorem{definition}[theorem]{Definition}
\theoremstyle{remark}

\usepackage[norelsize,boxed,linesnumbered,vlined]{algorithm2e}

\title{Deep Learning via Neural Energy Descent}

\author{%
  Wenrui Hao\\
  Department of Mathematics\\
   The Pennsylvania State University\\
    104 McAllister Building, University Park,\\
    State College, PA, 16802, USA\\
  \texttt{wxh64@psu.edu} \\
   \And
  Chunmei Wang \\
  Department of Mathematics \\
  University of Florida \\
1400 Stadium Rd, Gainesville, FL, 32611, USA\\
  \texttt{chunmei.wang@ufl.edu}
  \And 
  Xingjian Xu\\
  School of Mathematics and Statistics\\
  Lanzhou University\\
  No. 222 South Tianshui Road, \\Lanzhou, Gansu, 730000, China\\
  \texttt{xuxj2020@lzu.edu.cn}
   \And
   Haizhao Yang \\
   Department of Mathematics \\
   University of Maryland College Park \\
   4176 Campus Dr, College Park, MD, 20742,  USA\\
   \texttt{hzyang@umd.edu} \\
}

\begin{document}

\maketitle
\begin{abstract}
This paper proposes the Nerual Energy Descent (NED) via neural network evolution equations for a wide class of deep learning problems. We show that deep learning can be reformulated as the evolution of network parameters in an evolution equation and the steady state solution of the  partial differential equation (PDE) provides a solution to deep learning. This equation corresponds to a gradient descent flow of a variational problem and hence the proposed time-dependent PDE solves an energy minimization problem to obtain a global minimizer of deep learning. This gives a novel interpretation and solution to deep learning optimization. The computational complexity of the proposed energy descent method can be enhanced by randomly sampling the spatial domain of the PDE leading to an efficient NED. Numerical examples are provided to demonstrate the numerical advantage of NED over stochastic gradient descent (SGD).
\end{abstract}

\section{Introduction}
Learning a high-dimensional function or the solution of a high-dimensional and nonlinear partial differential equation (PDE) is ubiquitous and important in science and engineering  \cite{Lee2002,Ehrhardt2008,Yserentant2005,Gaikwad2009,Wales2003}. Generally speaking, there are no closed-form solutions to such high-dimensional problems and nonlinear PDEs that make the numerical solutions of such problems indispensable in real applications. Firstly, it is challenging to develop conventional numerical methods for high-dimensional problems since there is a curse of dimensionality in conventional discretization. 
Secondly, conventional numerical methods often rely on mesh generation  and  require  profound expertise and programming skills without the use of commercial software. In particular, it is challenging and time-consuming to implement the conventional methods for problems defined in complicated domains. 
 
As an efficient parametrization tool for high-dimensional functions \cite{Barron1993,E2019,bandlimit,Montanelli2019_2,SIEGEL2020313,HJKN19_814,Hutzenthaler2020_2,Shen4,Shen2021_2} with user-friendly software (e.g., TensorFlow and PyTorch), 
deep neural networks (DNNs) have become one of the most popular and important tools not only in computer science but also in other science and engineering problems with remarkable breakthroughs. 
For example, neural network-based PDE solvers dating back to the 1980s \cite{ODE1989,Lee1990,Gobovic1994,Dissanayake1994,lagaris1998artificial} were recently popularized for high-dimensional problems \cite{Han2018,Berg2018,Sirignano2018,Cai2019,Zang2020,Gu2020,RAISSI2019686,Khoo2017SolvingPP,Karpatne2017,Yucesan2020,Pan2020}. First of all, as a form of function approximation via the compositions of nonlinear functions \cite{IanYoshuaAaron2016}, DNNs are a mesh-free parametrization  and can efficiently approximate various high-dimensional solutions lessening the curse of dimensionality \cite{Barron1993,Hadrien,Weinan2019,Weinan2019APE,bandlimit,poggio2017,yarotsky19,Shen3,Montanelli2019_2,HJKN19_814,Shen4} and/or achieving exponential approximation rates  \cite{yarotsky2017,bandlimit,Shen3,DBLP:journals/corr/LiangS16,DBLP:journals/corr/abs-1807-00297,Opschoor2019,Shen4}. 
 Second, DNN parameters are identified via energy minimization from   the variational formulation. The computation of  variational formulation can be accelerated by SGD for a reasonably good minimizer. Developing the theoretical guarantee of these solvers has been an active research field recently \cite{DBLP:journals/corr/abs-1809-03062,Shin2020OnTC,LuoYang2020,Yulong21,Ziang22,Yuling}.  

Though learning high-dimensional functions admit the benefits mentioned above, the corresponding optimization problem is highly non-convex and thus challenging to solve for high accuracy. In the literature, extensive research has been conducted to improve the accuracy of neural network optimization. The following methods are listed as examples. Special neural networks are constructed to satisfy the initial/boundary conditions of the PDE aiming to simplify the optimization formulation and increase the accuracy \cite{712178,gu2020structure,Lyu2020EnforcingEB}. First-order methods are applied to reformulate high-order PDEs to reduce the difficulty of neural network optimization
\cite{cai2020deep,Lyu2020EnforcingEB}. New sampling strategies \cite{NakamuraZimmerer2019AdaptiveDL,Chen2019QuasiMonteCS} or important sampling \cite{pmlr-v107-liu20a,Gu2020} are proposed to  facilitate the convergence of neural network-based optimization. The special neural network structures or neural network solutions are constructed according to solution ansatz inspired by physical knowledge to significantly alleviate the training difficulty of neural network optimization including the use of oscillatory structures \cite{Cai2020}, multiscale structures \cite{Liu2020Jul}, and other spectral structures \cite{gu2020structure}. Combining neural network-based solvers and traditional iterative solvers, the hybrid algorithms  provide highly accurate solutions to low-dimensional nonlinear PDEs efficiently \cite{FP,huang2020int}.  

Besides the  application of deep learning to high-dimensional problems, deep learning also has advantages over traditional computational tools in certain low-dimensional problems. In computer vision and graphics, deep neural networks  as a mesh-free representation of objects, scene geometry and appearance,   
lead to  notable performance compared with traditional discrete representations. These DNNs, named ``coordinate-based" networks \cite{tancik2020fourier}, take low-dimensional coordinates as inputs
and output an object value of the shape, density, and/or color at the coordinate of the given input. This strategy is compelling in data compression and reconstruction  \cite{8953765,Jeruzalski2020NASANA,9157823,9010266,8954065,9156730,Saito2019PIFuPI,Sitzmann2019SceneRN,liang2021reproducing}. Similarly to the case of high-dimensional applications, it is also a challenging topic to achieve high accuracy in these applications.  It has been an active research direction to explore different neural network architectures and training strategies for highly accurate solutions to these problems.


This paper proposes  the neural energy descent (NED) algorithm for deep learning with theoretical justification in deep network approximation. We apply our NED for two applications: supervised learning and solving PDEs. The main philosophy of NED is to reformulate machine learning problems into the identification of the steady-state solution of an evolution equation, where the steady-state solution of this equation would provide a globally optimal solution to deep learning in the limit of infinite width. This equation corresponds to a gradient descent flow of a variational problem and hence the proposed time-dependent PDE solves an energy minimization problem to obtain the global minimizer of deep learning in the limit of infinite width. This gives a new interpretation and solution to deep learning optimization. We randomly sample the spatial domain of the PDE that leads to an efficient NED to enhance the computational complexity of the proposed energy descent method. Though the discretization of the evolution PDE, finite samples, and finite network width lead to numerical errors that would make the numerical solution deviate from the global minimizer of deep learning, from a series of numerical examples, we still observe the numerical advantage of NED over SGD.

This paper is organized as follows. In Section 2, two applications of NED are discussed.
Deep network approximation theory for NED is presented in Section 3.  In Section 4, we present the numerical implementation of NED. In Section 5, various numerical results are demonstrated to illustrate the advantages of NED over SGD. Finally, a conclusion is made in Section 6.

\section{Main Methodology of NED}
 In this section, we will  introduce the main methodology of NED to solve the machine learning problems: 1)  supervised learning using DNNs; 2) solving PDEs using DNNs.
 
\subsection{NED for Supervised Learning}
The learning task is to learn a DNN $U(x;\theta)$ with a parameter set $\theta$ such that $U(x;\theta)$ matches an unknown function $f(x)$ using a set of sample locations $\X =\{x_i\}_{i=1}^N$ and its corresponding function values $\Y=\{y_i=f(x_i)\}_{i=1}^N$. In NED, the corresponding evolution equation is the following ordinary differential equation (ODE):
	\begin{equation}
	\label{intro:problemODE}
	\left\{
		     \begin{array}{lr}
		    \partial_t u(x,t) + u(x,t) = f(x), &\quad   (x,t)\in \Omega\times[0,\infty], \\
		     u(x,0) = g(x), &\quad    x\in  \Omega,
		     \end{array}
	\right.
	\end{equation}
where $\X\subset\Omega \subset \mathbb{R}^d$, $f(x)$ is a given continuous function, and $g(x)$ is an arbitrary continuous function.  The solution of this ODE is $	u(x,t)=e^{- t}\left(  g(x)-f(x) \right)+f(x)$ and the steady state solution $u_s(x)$ satisfies $	u_s(x):=\lim_{t\rightarrow \infty} u(x,t)=f(x)$. Since $g(x)$ can be arbitrary, if the semi-discretization of $x$ via DNNs $U(x;\theta(t))$ is applied to \eqref{intro:problemODE}, we consider the following evolution equation of $\theta(t)$
		\begin{equation}
	\label{intro:problemODE2}
	\left\{
		     \begin{array}{lr}
		    \partial_t U(x;\theta(t)) + U(x;\theta(t)) = f(x), &\quad    (x,t)\in  \Omega\times[0,\infty], \\
		     \theta(0) = \theta_0, &
		     \end{array}
	\right.
	\end{equation}
which gives a gradient flow of $\theta(t)$ with an arbitrary initial condition $\theta_0$. The goal is to identify a steady-state solution $U(x;\theta^*):=\lim_{t\rightarrow \infty} U(x;\theta(t))= f(x)$ with $\theta^*=\lim_{t\rightarrow \infty} \theta(t)$ for $x\in\X\subset \Omega$ to solve the original supervised learning problem. The existence of a steady-state solution can be guaranteed if $f(x)$ is a neural network of the same size as $U(x;\theta(t))$; {otherwise, the existence would still be true in the limit of infinite network width.}

The evolution equation in \eqref{intro:problemODE} can be understood from a variational formulation. Consider the convex energy functional in $u(x,t)$ as follows
\begin{equation}
\label{eqn:fn1}
J_1(u(x,t);\mu)=\frac{1}{2}\int_\Omega (u(x,t)-f(x))^2 d\mu.
\end{equation}
\eqref{eqn:fn1} is a Lyapunov function for the evolution equation in \eqref{intro:problemODE} with a measure $\mu$ of $\Omega$. When the set of sample locations is a finite set $\mathcal{X}=\{x_i\}_{i=1}^N\subset\Omega$ and $\mu$ is the corresponding empirical measure of $\mathcal{X}$ in $\Omega$,  $J_1$ is the empirical loss function in supervised learning. When the sample set $\mathcal{X}=\Omega$ and $\mu$ is the standard Lebesgue measure of $\Omega$, $J_1$ is the population loss function in supervised learning. The following lemma derived by functional derivative shows that the gradient flow of \eqref{intro:problemODE} gives a solution $u(x,t)$ that minimizes the energy functional $J_1$ and the steady state solution $\lim_{t\rightarrow\infty}u(x,t)=f(x)$ is a global minimizer of the energy functional $J_1$.

\begin{lemma}
\label{lem:vf11}
Let $u(x,t)$ be the solution of \eqref{intro:problemODE} with an arbitrary initial condition $u(x,0)=g(x)$. There holds $\frac{\partial J_1(u(x,t);\mu)}{\partial t}=-\int_\Omega (\partial_t u(x,t))^2d\mu\leq 0$ and $\frac{\partial J_1(u(x,t);\mu)}{\partial t}=0$ if and only if $J_1(u(x,t);\mu)=0$.
\end{lemma}

\subsection{NED for Solving PDEs}	

The learning task is to learn a DNN $U(x;\theta)$ 
that solves a PDE.
For simplicity, we consider PDE problems with Dirichlet boundary conditions.   NED learns a DNN $U(x;\theta)$ matching an unknown function $u^*(x)$ that satisfies a given PDE on a sample set $\X$ in the PDE domain $\Omega\subset\R^d$ and matches a given function $h(x)$ on another sample set $\X_b$ on the boundary of the domain $\partial\Omega$. In a general setting, suppose the PDE is given as $\mL\left(u^*(x)\right)=f(u^*(x),x)$ for $x\in\X$, where $\mL$ is a given operator on $u^*(x)$ and $f(v,x)$ as a function of $v$ is known for $x\in\X$. This learning task in fact includes supervised learning when $\mL$ is an identity map and $f(v,x)=f(x)$ for any $v$. To be as general as possible, we consider the following time-dependent PDE in the NED framework:
	\begin{equation}
	\label{intro:problemPDE}
	\left\{
		     \begin{array}{lr}
		    \partial_t u(x,t) +\mL u= f(u,x), & (x,t) \in \Omega\times[0,\infty), \\
		 \partial_t u(x,t) +     u(x,t) =  h(x), &   (x,t)\in   \partial \Omega\times[0,\infty), \\
		     u(x,0) = g(x),  &   x \in   \Omega,\\
		     \end{array}
	\right.
	\end{equation}
	where $\X\subset \Omega \subset \R^d$, $\mL$ is a linear differential operator of order $n$, $u(x,t)\in \mathcal{W}^{n,p}(\Omega)\times C^m([0,\infty))$, $h\in \mathcal{W}^{n,p}(\partial\Omega)$, and $g\in \mathcal{W}^{n,p}(\Omega)$ is an arbitrary function. We assume that as $t\rightarrow \infty$, $u(x,t)$ converges to a steady state solution of \eqref{intro:problemPDE}, denoted by $u_s(x)$. It is easy to check $u_s(x)$ satisfies $\mL\left(u_s(x)\right)=f(u_s(x),x)$ in $\Omega$ and $u_s(x)=h(x)$ on $\partial\Omega$. When the semi-discretization via DNNs is applied to $u(x,t)$, we have
		\begin{equation}
	\label{intro:problemPDE2}
	\left\{
		     \begin{array}{lr}
		    \partial_t U(x;\theta(t)) +\mL U(x;\theta(t))= f(U(x;\theta(t)),x), &  (x,t)\in   \Omega\times[0,\infty), \\
		 \partial_t U(x;\theta(t)) +     U(x;\theta(t)) = h(x), & (x,t) \in   \partial \Omega\times[0,\infty), \\
		     \theta(0) = \theta_0,  &\\
		     \end{array}
	\right.
	\end{equation}
which gives a gradient flow of $\theta(t)$ with an arbitrary initial condition $\theta_0$. The goal is to identify a steady state solution $U(x;\theta^*)=\lim_{t\rightarrow \infty}U(x;\theta(t))$ of \eqref{intro:problemPDE2} to learn the unknown function  using DNNs satisfying $\mL\left(U(x;\theta^*)\right)= f(U(x;\theta^*),x)$ in $\Omega$ and $U(x;\theta^*)= h(x)$ on $\partial\Omega$. 

We can incorporate the boundary condition $U(x;\theta))=h(x)$ into the discretization by using  
$U(x;\theta)=L_D(x)N(x;\theta)+\ell(x)$, where $N(x;\theta)$ is a neural network to be trained, $L_D(x)$ is a function measuring the distance from $x$ to the boundary $\partial\Omega$, and $\ell(x)$ is an arbitrary function equal to $h(x)$ on $\partial\Omega$. Thus $U(x;\theta)=h(x)$ automatically holds true when $x\in \partial \Omega$.
As a result, the semi-discretization in (\ref{intro:problemPDE2}) is further simplified to 

		\begin{equation}
	\label{intro:problemPDE21}
	\left\{
		     \begin{array}{lr}
		    \partial_t U(x;\theta(t)) +\mL U(x;\theta(t))= f(U(x;\theta(t)),x), &   (x,t) \in   \Omega\times[0,\infty), \\
		     \theta(0) = \theta_0.  &\\
		     \end{array}
	\right.
	\end{equation}

The evolution equation in \eqref{intro:problemPDE} is complicated when the operator $\mL$ and the function $f(u,x)$ are nonlinear in $u$. The evolution equation in \eqref{intro:problemPDE} can also be understood from a variational formulation when $\mL$ and $f$ are linear in $u$. Let us start with the linear case, e.g.,
\begin{equation}\label{eqn:L}
\mL(u)=-\sum_{i,j=1}^d \frac{\partial}{\partial x_i}\left( A_{ij}(x) \frac{\partial u}{\partial x_j}\right)+c(x) u,
\end{equation}
where a matrix $A(x)=( A_{ij}(x))_{i,j=1}^d$. Without loss of generality, we assume $f(u,x)$ is independent of $u$. For simplicity of notations, we denote $f(u,x)=f(x)$. Consider the energy functionals in $u(x,t)$ as follows
\begin{equation}
\label{eqn:fn2}
J_2(u(x,t);\mu)=\frac{1}{2}\int_\Omega \left( \mL u\right)u(x,t)d\mu-\int_\Omega f(x)u(x,t)d\mu,
\end{equation}
and
\begin{equation}
\label{eqn:fn3}
J_3(u(x,t);\mu)=\frac{1}{2}\int_{\partial\Omega} (u(x,t)-h(x))^2 d\mu,
\end{equation}
which are the Lyapunov  functions for the first and second evolution equations in \eqref{intro:problemPDE}, respectively. Here $\mu$ stands for the standard Lebesgue measure if $\mathcal{X}=\Omega$ and $\mathcal{X}_b=\partial\Omega$; otherwise, it represents the empirical measure corresponding to discrete samples. Similarly to the discussion for supervised learning, the gradient flow of \eqref{intro:problemPDE} gives a solution $u(x,t)$ that minimizes the energy functional $J(u;\mu):=J_2(u;\mu)+J_3(u;\mu)$. If the operator $\mL$ is positive semidefinite, $J(u;\mu)$ is convex and the steady state solution $\lim_{t\rightarrow\infty}u(x,t)=f(x)$ is a global minimizer of $J(u;\mu)$.
A similar discussion can also be applied to \eqref{intro:problemPDE21}.

\section{Deep Network Approximation Theory for NED}
\label{sec:app}
As a first step to justify NED, we prove the existence of a neural network $U(x;\theta(t))$ in the semi-discretization in \eqref{intro:problemODE2} and \eqref{intro:problemPDE2} to approximate the solution $u(x,t)$ in the original model in \eqref{intro:problemODE} and \eqref{intro:problemPDE} uniformly well for all $t$, respectively. The overview of the approximation analysis is summarized in Figure \ref{fig:err2}. In particular, we are interested in  the convergence rate of the approximation in terms of the width and depth of DNNs. Furthermore, we pay specific attention to the smoothness of network parameter $\theta(t)$, which is required in \eqref{intro:problemODE2} and \eqref{intro:problemPDE2}. The existence is a result of a more general theory on deep network approximation in the Sobolev space. We will present our main theorems in this section and their proofs can be found in the appendix.

	\begin{figure}[!ht]
		\centering
		\includegraphics[width=0.8\linewidth]{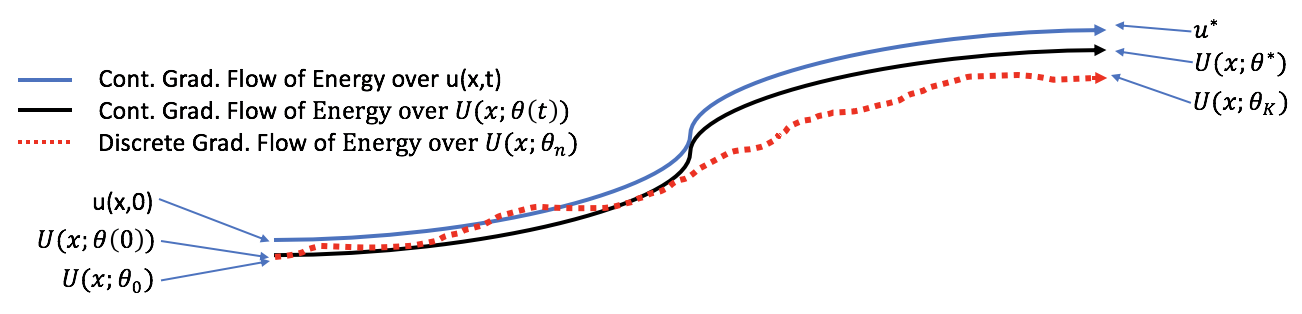}
		\caption{The overview of the deep network approximation analysis of NED. Continuous energy functional minimization provides a solution path $u(x,t)$ to a global minimizer of a learning problem as $t\rightarrow \infty$. We show that there exists a solution path $U(x;\theta(t))$ in the form of DNNs in a small neighborhood of $u(x,t)$. The approximation error of $U(x;\theta(t))$ to approximate $u(x,t)$ is characterized in terms of the DNN width $N$ and depth $L$. The goal of our numerical scheme is to generate a sequence of parameter sets $\{\theta_n\}_{n=1}^K$ such that $U(x;\theta_K)$ can approximate $U(x;\theta^*)=\lim_{t\rightarrow \infty} U(x;\theta(t))$ within a well-controlled error visualized by red dots. }
		\label{fig:err2}
	\end{figure}

Let us start with the introduction of DNNs as a nonlinear function parametrization and present the corresponding approximation theory later in this section. To make our discussion as general as possible, we will focus on the fully connected neural network (FNN) that includes various practical network structures as its special cases.

\begin{definition}\label{def:FNN}
Given $\bm{N}=[N_1,N_2,\dots,N_{L+1}]\in \N_+^d$ and $L\in \N_+$, a function $\Phi(x;\theta)$ is a fully connected neural network of $x\in\R^d$ with a parameter set $\theta$, width $\bm{N}$, and depth $L$ if $\Phi(x;\theta)={h}_{L+1}(x;\theta)$ is defined recursively via the composition of $L$ nonlinear functions as follows:
  \[
   h_i :=	\left\{
		     \begin{array}{lr}
		   {W}_i x+ {b}_i, & \text{if } i=1, \\
		   {W}_i \tilde{{h}}_{i-1} + {b}_i, & \text{if } i=2,\dots, L+1,
		     \end{array}
	\right.
	\]
      where
    \[
       \tilde{{h}}_i=\sigma_{s_i}({h}_i),\quad \tn{for $i=1$, $\dots$, $L$,}
    \]
 ${W}_i\in \R^{N_{i}\times N_{i-1}}$ and ${b}_i\in \R^{N_i}$ for $i=1$, $\dots$, $L+1$ are the weight matrix and the bias vector in the $i$-th linear transform in $\Phi$, respectively, $\sigma_{s_i}$ is called the nonlinear activation function with a parameter set $s_i$ and its action to a vector is entry-wise, $\theta$ is the union of all parameters in $\{W_i\}_{i=1}^{L+1}$, $\{b_i\}_{i=1}^{L+1}$, and $\{s_i\}_{i=1}^{L}$.
\end{definition}
For simplicity, we will let $N_1=N_2=\dots N_{L+1}=N\in \N_+$ throughout this paper and say $N$ is the width of the FNN.

\begin{definition}\label{def:s1}
If the activation function of an FNN is chosen as the rectified linear unit (ReLU), i.e., $\sigma(x)=\max\{x,0\}$, the corresponding FNN is called a $\sigma_1$-$\tn{NN}$ in this paper.
\end{definition}

It is obvious that $\sigma_1$-$\tn{NN}$s are piecewise linear functions and, hence, are capable of approximating classifier functions in classification problems as a special case of supervised learning problems, which can be solved via the evolution equation in \eqref{intro:problemODE}. However, the second derivative of $\sigma_1$-$\tn{NN}$s vanishes and, hence, $\sigma_1$-$\tn{NN}$s are not suitable for a more general learning task that can be solved by the evolution equation in \eqref{intro:problemPDE}. This motivates the introduction of $\sigma_2$-$\tn{NN}$s as follows.

\begin{definition}\label{def:s2}
If the activation function of an FNN is chosen as $\sigma_{a,b}(x)=a\odot \max\{x,0\}+b\odot x\odot\max\{x,0\}$ with two vectors $a\in \R^n$ and $b\in\R^n$ as parameters for $x\in\R^n$, where $\odot$ means the entry-wise multiplication, then the corresponding FNN is called a $\sigma_2$-$\tn{NN}$ in this paper.
\end{definition}

We are now ready to introduce our main theorems in the approximation theory for NED, e.g., $u(x,t)\approx U(x;\theta(t))$ in \eqref{intro:problemODE} and \eqref{intro:problemPDE} with $\theta(t)$ satisfying certain continuity. The approximation rates in our main theorems are not sharp. Following the ideas in \cite{Shen1,Shen2,Shen3,yarotsky18a,yarotsky19,Ingo,Opschoor2019}, nearly optimal approximation rates can be derived in the Sobolev space. However, $\theta(t)$ in these nearly optimal FNNs cannot be continuous in $t$ by the theory of optimal nonlinear approximation \cite{Ron} (Theorem 4.2), making the semi-discretization of evolution equations in \eqref{intro:problemODE2} and \eqref{intro:problemPDE2} invalid.
The proofs of these theorems can be found in the Appendix. Without loss of generality, we assume $\Omega\times \Omega_t=[0,1]^d\times [0,\infty)$.

Theorem \ref{thm:ap0} below quantitatively characterizes the approximation capacity of $\sigma_1$-$\tn{NN}$s in the $C^0(\Omega)\times C^m(\Omega_t)$ space so as to justify the semi-discretization schemes in \eqref{intro:problemODE2} for supervised learning, since ReLU activation functions are popular for supervised learning.

\begin{theorem}
\label{thm:ap0}
Let $m\in \N$, $d\in \N_+$, $1\leq p\leq \infty$. For any $N,L\in \N$ such that $N\geq 2d+14$ and $L\geq d^2-d+1$, any $f(x,t)\in C^0(\Omega)\times C^m(\Omega_t)$,  there exists a $\sigma_1$-$\tn{NN}$ $\Phi(x;\theta(t))$ such that
\[
\|\Phi(x;\theta(t))-f(x,t)\|_{L^p((0,1)^d)}
\leq 3d\cdot\omega_{f(x,t)}\left(\frac{1}{\left( \left\lfloor \frac{N-2-2d}{12} \right\rfloor \left\lfloor \frac{L}{d^2-d+1} \right \rfloor \right)^{1/d}-1}\right),
\]
where $\omega_f(\cdot)$ is the modulus of continuity of a function $f$ defined via
\begin{equation*}
\omega_f(r):= \sup\big\{|f({x})-f({y})|:{x},{y}\in [0,1]^d,\ \|{x}-{y}\|_2\le r\big\},\quad \tn{for any $r\ge 0$}.
\end{equation*}
Furthermore, $\theta(t)$ is in $C^m(\Omega_t)$.
\end{theorem}

As a simple corollary, we can characterize the approximation error of the semi-discretization scheme to the model in \eqref{intro:problemODE} when we apply $\sigma_1$-$\tn{NN}$s as follows.

\begin{corollary}
\label{col:ap0}
Let $m\in\N$, $d\in\N_+$, $1\leq p\leq \infty$. Suppose $u(x,t)\in C^0(\Omega)\times C^m(\Omega_t)$ is a solution of \eqref{intro:problemODE}, and $u(x,t)$ is Lipschitz continuous with a Lipschitz constant $\nu(t)$ for a fixed $t$.  For any $N,L\in \N$ such that $N\geq 2d+14$ and $L\geq d^2-d+1$, there exists a $\sigma_1$-$\tn{NN}$ $\Phi(x;\theta(t))$ with $\theta(t)$ in $C^m(\Omega_t)$ such that
\[
\|\Phi(x;\theta(t))-u(x,t)\|_{L^p(\Omega)}
\leq \frac{3d\nu(t)}{\left( \left\lfloor \frac{N-2-2d}{12} \right\rfloor \left\lfloor \frac{L}{d^2-d+1} \right \rfloor \right)^{1/d}-1}
\]
for a fixed $t$.
\end{corollary}

As an immediate result of Corollary \ref{col:ap0}, if $\nu(t)$ is uniformly bounded for $t\in \Omega_t$, then the solution of \eqref{intro:problemODE2} can approximate the solution of \eqref{intro:problemODE} with an error $\OO(\left(N L  \right)^{-1/d})$, uniformly in $t$ measured in the $L^p$-norm for $p\in[1,\infty]$, and the steady state solution $U(x;\theta^*):=\lim_{t\rightarrow \infty}U(x;\theta(t))$ solves the corresponding learning task with an error $\OO(\left(N L  \right)^{-1/d})$ measured in the same norm.

Theorem \ref{thm:ap2} below quantitatively characterizes the approximation capacity of $\sigma_2$-$\tn{NN}$s in the Sobolev spaces so as to justify their applications to the semi-discretization schemes in \eqref{intro:problemODE2} and \eqref{intro:problemPDE2}, respectively. Typically, $\sigma_2$-$\tn{NN}$s are used for solving PDEs since $\sigma_1$-$\tn{NN}$s vanish after high-order derivatives. Other activation functions that do not vanish after high-order derivatives could also be applied. For simplicity, we only focus on $\sigma_2$-$\tn{NN}$s in our approximation theory.

\begin{theorem}\label{thm:ap2}
Let $m,d\in \N$, $n\in \N_+$, $1\leq p\leq \infty$, and $s\in \N$ with $0\leq s\leq n-1$. For any $\epsilon\in(0,\frac{1}{2})$ and any $f(\xx,t)\in \Wnp(\Omega)\times C^m(\Omega_t)$, we have  
\begin{enumerate}[label=(\roman*)]
\item There exists a $\sigma_2$-$\tn{NN}$ $\Phi_1(\xx;\xth_1(t))$ and $C_s =C_s(n,d,p)$ such that
\begin{enumerate}
\item $\|\Phi_1(\xx;\xth_1(t))-f(\xx,t)\|_{\Wsp((0,1)^d)}\leq \epsilon $;
\item The depth of $\Phi_1(\xx;\xth_1(t))$ is at most $1+(4+2(n-s))d^{n-s}$;
\item The width of $\Phi_1(\xx;\xth_1(t))$ is at most $2d+2+4\left(  \frac{C_s \|f\|_{\Wnp((0,1)^d)}}{\epsilon} +2^{n-s} \right)^{d/(n-s)}$.
\end{enumerate}
\item For any $ {N}, {L}\in \N$ such that $ {N}\geq 2d+6$ and $ {L}\geq (4+2(n-2))d^{n-s}$, there exists a $\sigma_2$-$\tn{NN}$ $\Phi_2(\xx;\xth_2(t))$ such that
\[
\|\Phi_2(\xx;\xth_2(t))-f(\xx,t)\|_{\Wsp((0,1)^d)}
\leq \frac{\bar{C}_s  \|f\|_{\Wnp((0,1)^d)} }{\left( {N} {L} \right)^{(n-s)/d}},
\]
where $\bar{C}_s =\bar{C}_s(n,d,p)$.
\end{enumerate}
Furthermore, $\xth_1(t)$ and $\xth_2(t)$ are in $C^m(\Omega_t)$.
\end{theorem}

As a corollary, we can characterize the approximation error of the semi-discretization scheme applied to solve PDEs when we apply $\sigma_2$-$\tn{NN}$s as follows.


\begin{corollary}
\label{col:ap2}
Let $m,d\in\N$, $n\in\N_+$, $1\leq p\leq \infty$, and $s\in\N$ with $0\leq s\leq n-1$. Suppose  $u(x,t)\in\Wnp(\Omega)\times C^m(\Omega_t)$ is a solution of \eqref{intro:problemPDE}, and $u(x,t)$ is in $C^s(\overline{\Omega})$ for any $t\in \Omega_t$.
For any $N,L\in \N$ such that $N\geq 2d+6$ and $L\geq (4+2(n-2))d^{n-s}$, for any $s$-th order differential operator $\mL$ with all coefficients absolutely bounded by $\beta$, there exists a $\sigma_2$-$\tn{NN}$ $\Phi(x;\theta(t))$ such that
\[
\|\Phi(x;\theta(t))-u(x,t)\|_{\Wsp(\Omega)}
\leq \frac{{C}_{1,s} \|u(x,t)\|_{\Wnp(\Omega)} }{\left(N L  \right)^{(n-s)/d}},
\]
\[
\|(\partial_t+\mL)(\Phi(x;\theta(t))-u(x,t))\|_{L^p(\Omega)}
\leq \frac{{C}_{2,s}( \|\partial_t u\|_{\Wnp(\Omega)} +\|u\|_{\Wnp(\Omega)})}{\left(N L  \right)^{(n-s)/d}},
\]
\[
\|\Phi(x;\theta(t))-u(x,t)\|_{\Wsp(\partial\Omega)}
\leq \frac{{C}_{3,s} \|u(x,t)\|_{\WW^{n,\infty}(\Omega)} }{\left(N L  \right)^{(n-s)/d}},
\]
with $\theta(t)$ in $C^m(\Omega_t)$, where ${C}_{1,s}={C}_{s}(n,d,p)$, ${C}_{2,s}={C}_{2,s}(n,d,p,\beta)$, and ${C}_{3,s}={C}_{3,s}(n,d,p)$.
\end{corollary}

\begin{proof}
The error bound on $\Omega$ is a direct result of Theorem \ref{thm:ap2}. The error bound on $\partial \Omega$ holds by the fact that
\begin{eqnarray*}
\|\Phi_j(x;\theta_j(t))-u(x,t)\|_{\Wsp(\partial\Omega)}&\leq & c \|\Phi_j(x;\theta_j(t))-u(x,t)\|_{\WW^{s,\infty}(\partial\Omega)}\\
(u(x,t) \in C^s(\bar{\Omega})\text{ for any }t) &\leq & c \|\Phi_j(x;\theta_j(t))-u(x,t)\|_{\WW^{s,\infty}(\Omega)}\\
(\text{Theorem \ref{thm:ap2}}) &\leq & \frac{{C}_{3,s} \|u(x,t)\|_{\WW^{n,\infty}(\Omega)} }{\left(N L  \right)^{(n-s)/d}},
\end{eqnarray*}
where $c=c(s,d,p)$ and ${C}_{3,s}={C}_{3,s}(n,d,p)$.
\end{proof}

As an immediate result of Corollary \ref{col:ap2}, if $ \|u(x,t)\|_{\Wnp(\Omega)}$ and $ \|u(x,t)\|_{\WW^{n,\infty}(\Omega)}$ are uniformly bounded for $t\in \Omega_t$, then there exists a solution of \eqref{intro:problemPDE2} (and \eqref{intro:problemODE2}) approximating the solution of \eqref{intro:problemPDE} (and \eqref{intro:problemODE}) with an error $\OO(\left(N L  \right)^{-(n-s)/d})$, uniformly in $t$ measured in the $\Wsp$-norm, and the steady state solution $U(x;\theta^*):=\lim_{t\rightarrow \infty}U(x;\theta(t))$ solves the corresponding learning task with an error $\OO(\left(N L  \right)^{-(n-s)/d})$ measured in the $\Wsp$-norm.

\section{Numerical Implementation of NED}
\label{sec:num}

In this section, we will introduce the numerical implementation of NED for the evolution equation \eqref{intro:problemODE2} for supervised learning and the evolution equation \eqref{intro:problemPDE2} for solving PDEs. 

\subsection{Numerical Implementation of NED for Supervised Learning}
\label{sub:NED1}
In the supervised learning case when the sample set $\X=\Omega$, we will randomly sample $N$ locations in $\Omega$ to discretize the domain $\Omega$ at each time slide $t$ and denote the sample set as $\X_t$. When $\X$ is a set of finite locations, we let $\X_t=\X$.

After the semi-discretization of \eqref{intro:problemODE} with $\X_t$, the resulting evolution equation of $\theta(t)$ in \eqref{intro:problemODE2} becomes
	\begin{equation}
		\label{odeODE}
\nabla_\theta U(\X_t;\theta(t))\dot{\theta}(t)=\mathbf{R}(\X_t;\theta(t)),
	\end{equation}
	with an initial condition $\theta(0)=\theta_0$, where $\mathbf{R}(\X_t;\theta(t)) = f(\X_t) - U(\X_t;\theta(t))$
	is a nonlinear mapping from $\theta(t)\in \mathbb{R}^{|\theta|}$ to $ \mathbf{R}(\X_t;\theta(t)) \in \mathbb{R}^N$ measuring the regression error of the DNN on $\X_t$, where $|\theta|$ denotes the total number of parameters in $\theta$ and $N$ is the total number of training data in $\X_t$. Note that
\[
\nabla_\theta U(\X_t;\theta(t)) = - \nabla_\theta\mathbf{R}(\X_t;\theta(t)).
\]
 Hence, we get
\begin{equation}\label{odematODE}
 - \nabla_\theta\mathbf{R}(\X_t;\theta(t))\dot{\theta}(t)=\mathbf{R}(\X_t;\theta(t)),
\end{equation}
where $ - \nabla_\theta\mathbf{R}(\X_t;\theta(t))\in\mathbb{R}^{N\times|\theta|}$, $\dot{\theta}(t)\in\mathbb{R}^{|\theta|}$, and $\mathbf{R}(\X_t;\theta(t))\in\mathbb{R}^N$. The gradient flow of $\theta(t)$ is given via a solution of a possibly under-determined linear system in \eqref{odematODE} (e.g., when $N<|\theta|$) and hence the gradient flow in NED may not be unique. It is application-dependent to choose a gradient flow that best fits a specific task. Without loss of generality, we propose to solve the following least square problem
   \begin{eqnarray}\label{eqn:ls}
   \min_{\theta(t)}&& \|\dot{\theta}(t)\|_*+
   \| \mathbf{R}(\X_t;\theta(t)) + \nabla_\theta\mathbf{R}(\X_{t};\theta(t))  \dot{\theta}(t) \| _+
   \end{eqnarray}
   with the constraint $\theta(0)=\theta_0$ and two appropriate norms $\|\cdot\|_*$ and $\|\cdot\|_+$ to determine the gradient flow of $\theta(t)$.

If $\|\cdot\|_*=\|\cdot\|_2$, the least square problem \eqref{eqn:ls} has an explicit solution, which is equivalent to solving \eqref{odematODE} using the Moore-Penrose inverse to obtain the following gradient flow
\begin{equation}\label{odeflowUnderPara}
\dot{\theta}(t)=-\left( \nabla_\theta\mathbf{R}(\X_t;\theta(t))^* \nabla_\theta\mathbf{R}(\X_t;\theta(t)) \right)^{+} \nabla_\theta\mathbf{R}(\X_t;\theta(t))^*\mathbf{R}(\X_t;\theta(t)),
\end{equation}
where $^*$ represents the conjugate transpose and $^+$ means the Moore-Penrose inverse.

Now we discretize the implicit gradient flow in \eqref{eqn:ls} in time $t$ with $\{t_k\}_{k=0}^K$ and identify $\{\theta_k\}_{k=0}^K$ via a $q$-th order explicit Runge-Kutta method. When $q=1$, the Runge-Kutta method becomes the forward Euler method. To this end, we propose Algorithm \ref{alg1ODE} below. For simplicity, we will present our algorithm for $q=1$ as an example.

\begin{algorithm}
 \KwData{A time interval $[0,T]$ and its uniform discretization points $\{t_k\}_{k=0}^K$. Training data $\{\X_{t_k},f(\X_{t_k})\}_{k=0}^K$. A random initial parameter set $\theta_0$.}
 \KwResult{$\{\theta_k\}_{k=0}^K$ }

  \For{$k=0,1,\dots,K-1$}{

  Evaluate $\nabla_\theta\mathbf{R}(\X_{t_k};\theta_k) $ and $b=
  \mathbf{R}(\X_{t_k};\theta_k)$\;

   Solve the least square problem
   \begin{eqnarray}\label{eqn:ls2}
   \min_{\alpha}&& \|\alpha\|_*\\
   -\nabla_\theta\mathbf{R}(\X_{t_k};\theta_k)  \alpha&=& \mathbf{R}(\X_{t_k};\theta_k)\nonumber
   \end{eqnarray}
   with an appropriate norm $\|\cdot\|_*$ as regularization\;

   Compute $\theta_{k+1}= \theta_k + \frac{T}{K} \alpha $ \; 
  }

 \caption{NED for supervised learning by using the forward Euler method. }
 \label{alg1ODE}
\end{algorithm}

\subsubsection{Stochastic Gradient Descent Method as a Comparison}
Note that the solution of \eqref{intro:problemODE} approaches to the solution of the following regression problem when $t\to\infty$,
	\begin{equation}
	\label{problem2ODE}
\theta=\arg\min \mathbb{E}_{\X} \mathbf{R}^2(\X;\theta).
\end{equation}
The gradient flow of the SGD reads
\begin{equation}\label{odeflow2}
\dot{\theta}=\frac{2}{N}\nabla_\theta\mathbf{R}(\X;\theta(t))^* \mathbf{R}(\X;\theta(t)).
\end{equation}
Therefore, the SGD updates in the range space of $
\nabla_\theta\mathbf{R}(\X;\theta(t))$ while our algorithm updates the similar gradient flow $\nabla_\theta\mathbf{R}(\X;\theta(t))$   using the Moore-Penrose inverse. The  numerical results in Section 5  will demonstrate that the method in Algorithm \ref{alg1ODE} is better.

\subsection{Numerical NED for Solving PDEs}
\label{sub:NED2}
After the semi-discretization of \eqref{intro:problemPDE}, the resulting evolution equation of $\theta(t)$ in \eqref{intro:problemPDE2} is equivalent to
	\begin{equation}
		\label{ode}
\nabla_\theta U(x;\theta(t))\dot{\theta}(t)=\mathbf{R}(x;\theta(t))\text{ for }(x,t)\in \Omega
\times[0,T]	\end{equation}
	with an initial condition
	\begin{equation}\label{int}
	\theta(0)=\theta_0 \text{ such that } U(x;\theta_0)=u_0(x),
	\end{equation}
	and a boundary condition
	\begin{equation}\label{bd}
	\mathbf{B}(x;\theta(t)):= U(x;\theta(t)) - h(x)=0\text{ for } x\in\partial\Omega.
	\end{equation}
Here \[\mathbf{R}(x;\theta(t))=	\Delta U(x;\theta(t)) + f(U(x;\theta(t)))\]
is a nonlinear mapping from $\theta(t)\in \mathbb{R}^{|\theta|}$ to $ \mathbf{R}(x;\theta(t)) \in \mathbb{R}$ for a fixed sample $x$.

We choose a set of $N$ samples in $\Omega$ denoted by $\X_{in}=\{x_i\}_{i=1}^N$ and a set of $M$ samples in $\partial\Omega$ denoted by $\X_{bd}=\{x_j\}_{j=1}^M$. 
Substituting $\X_{in}$ to \eqref{ode} yields $\nabla_\theta U(x_i;\theta(t))\dot{\theta}(t)=\mathbf{R}(x_i;\theta(t))$ for $i=1,\dots,N$; i.e.,
\begin{equation}\label{odemat}
\nabla_\theta U(\X_{in};\theta(t))\dot{\theta}(t)=\mathbf{R}(\X_{in};\theta(t)),
\end{equation}
where $\nabla_\theta U(\X_{in};\theta(t))\in\mathbb{R}^{N\times|\theta|}$, $\dot{\theta}(t)\in\mathbb{R}^{|\theta|}$, and $\mathbf{R}(\X_{in};\theta(t))\in\mathbb{R}^N$.

By differentiating $\mathbf{B}(x;\theta(t))$ with respect to $t$, we have $\nabla_\theta U(x_j;\theta(t)) \dot{\theta}(t) =0$ for $ j=1,\cdots, M$; i.e.,
\begin{equation}
\nabla_\theta U(\X_{bd};\theta(t)) \dot{\theta}(t) =0 ,
\label{ODE3}
\end{equation}
where $\nabla_\theta\mathbf{B}(\X_{bd};\theta(t))=\nabla_\theta U(\X_{bd};\theta(t))\in \mathbb{R}^{M\times |\theta|}$, the basis vectors of {whose} kernel space form a matrix $D\in\mathbb{R}^{|\theta|\times r}$ for some $r\in\mathbb{N}$.

{Instead of differentiating the boundary condition}, we should solve
\begin{equation}
\nabla_\theta U(\X_{bd};\theta(t)) \dot{\theta}(t) =-\mathbf{B}(x;\theta(t)).
\label{ODE2}
\end{equation}

Hence, after discretizing \eqref{ode} and \eqref{bd} with point sets $\X_{in}$ and $\X_{bd}$, we need to solve both (\ref{odemat}) and (\ref{ODE2}) with the initial condition in \eqref{int}.

Now we discretize in time $t$ with $\{t_k\}_{k=0}^K$ and identify $\{\theta_k\}_k$ to obtain our final solution $U(x,\theta_k)$. To this end, we propose Algorithm \ref{alg1} below.  Actually similar to Algorithm \ref{alg1ODE}, the $q$-th order explicit Runge-Kutta method can also be applied to Algorithm \ref{alg1} to update and identify the series $\{\theta_k\}_{k=0}^K$. The following Algorithm \ref{alg1} is also presented as the example when $q=1$ (forward Euler method).

\begin{algorithm}
 \KwData{$f(x)$, $u_0(x)$, and $\tilde{u}(x)$. A time interval $[0,T]$ and its uniform discretization points $\{t_k\}_{k=0}^K$.}
 \KwResult{$\{\theta_k\}_{k=0}^K$ }
 Apply deep learning to find $\theta_0$ and the corresponding DNN $U(x;\theta_0)$ such that $U(x;\theta_0)\approx u_0(x)$.

  \For{$k=0,1,\dots,K-1$}{
   Randomly select spatial discretization point sets $\X_{in}\subset \Omega$ and $\X_{bd}\subset\partial\Omega$.

  Evaluate $\mathcal{A}=\begin{pmatrix}
  \nabla_\theta U(\X_{in};\theta_k)\\
  \nabla_\theta U(\X_{bd};\theta_k)
  \end{pmatrix}$ and $b=\begin{pmatrix}
  \mathbf{R}(\X_{in};\theta_k)\\
  -\mathbf{B}(\X_{bd};\theta_k)
  \end{pmatrix}$\;

Find $\alpha$ such that it solves the least square problem 
   $\mathcal{A}\alpha\approx b$\;
   
   
   Compute $\theta_{k+1}= \theta_k + \frac{T}{K} \alpha$\;
  }

 \caption{NED for solving PDE problems by using the forward Euler method. }
 \label{alg1}
\end{algorithm}

\subsubsection{Stochastic  Gradient Descent Method as a Comparison}

Note that the steady state solution of \eqref{intro:problemPDE} is the solution of the following Laplace's equation:
	\begin{equation}
	\label{problem2}
	\left\{
		     \begin{array}{lr}
		    - \Delta u = f(u), & \text{in } \Omega, \\
		     u(x) = \tilde{u}(x), & \text{on } \partial \Omega.
		     \end{array}
	\right.
	\end{equation}
The optimization problem becomes
\[
\theta=\arg\min\frac{1}{N}\sum_{i=1}^N\mathbf{R}^2(x_i;\theta)+\frac{\lambda}{M} \sum_{j=1}^M\mathbf{B}^2(x_j;\theta).
\]
The gradient flow of the SGD is
\[
\dot{\theta}=\frac{2}{N}\nabla_\theta\mathbf{R}(\X_{in};\theta)^* \mathbf{R}(\X_{in};\theta) +\frac{2\lambda}{M}\nabla_\theta\mathbf{B}(\X_{bd};\theta)^*  \mathbf{B}(\X_{bd};\theta).
\]
Therefore, the SGD updates in the range space of $\begin{pmatrix}
\nabla_\theta\mathbf{R}(\X_{in};\theta), \nabla_\theta\mathbf{B}(\X_{bd};\theta)
\end{pmatrix}$ while Algorithm \ref{alg1} updates $\nabla_\theta U(\X_{in};\theta)$.

\section{Numerical Results}
\label{sec:results}
 In this section, several numerical examples are provided to show the numerical performance of NED. We shall compare the numerical performance of our NED with SGD. In NED, two explicit numerical methods are employed to solve the ODE in (\ref{odeflowUnderPara}).  One is the explicit forward Euler (FE) method; i.e.,
	\begin{equation}
\theta^{k+1}=\theta^k+\gamma(\theta^k)\eta.
	\end{equation}
The other one is the 2nd order Runge-Kutta (RK2) method; i.e.,
	\begin{equation}
\phi_1=\eta \gamma(\theta^k),~\phi_2=\eta \gamma(\theta^k+\phi_1/2),~\hbox{and}~\theta^{k+1}=\theta^k+\phi_2,
	\end{equation}
where $\eta$ is the time stepsize or the learning rate. It is well-known that the FE method has a first-order accuracy while the RK2 method has a second-order accuracy. 

For simplicity,   we will construct a neural network solution that satisfies the boundary conditions automatically \cite{gu2020structure} in all numerical experiments. We thus define a relative  ${L}^{2}$ error as follows
	\begin{equation}
e_{u} =\left(\frac{\sum_{i=1}^{N}|U(x_i;\theta(t))-u_s(x_i)|^{2}}{\sum_{i=1}^{N}|u_s(x_i)|^{2}}\right)^{\frac{1}{2}},
	\end{equation}
 where $\left\{x_{i}\right\}_{m=1}^{N}$ is the set of the random sample points uniformly distributed in  $\X_{t}$, and $u_{s}(x_i)$ is the steady state solution for regression and PDE problems.

The setting for all numerical examples is summarized as follows.
\begin{itemize}
  \item \textbf{Environment.}
  The experiments are performed in Python 3.8 environment. We utilize the Pytorch library for the implementation of the NED method and CUDA  11.6 toolkit for GPU-based parallel computing. All numerical examples are implemented on a desktop.
  \item \textbf{Learning rate.} The learning rate is set to be
  \begin{equation}\label{lr}
    \tau_n = q \tau_{0}(\cos(\pi \frac{n}{K})+1),
  \end{equation}
  where $\tau_n$ is the learning rate in the $n$-th iteration, $q$ is a parameter set to be $\frac{1}{2}$, $K$ is the number of all iterations, and $\tau_{0}$ is an initial learning rate, which will be specified in the numerical experiments. 
  \item \textbf{Network setting.} For supervised learning, we construct the fully connected neural network (FNN) to approximate the solution and use the $ReLU$ activation function if no specialization. All weights and biases in the $l$-th layer are initialized via an uniform distribution $U(-\sqrt{N_{l-1}},\sqrt{N_{l-1}})$, where $N_{l-1}$ is the width of the $l-1$-th layer.  For solving PDEs, we use the Resnet with the fixed block to approximate the steady state solution and employ $ReLU^{3}$ as the activation function if no specialization. 
  \item \textbf{Numbers of samples.}  The numbers of samples for $\X_{in}$ are randomly selected in the domain $\Omega$. In every epoch, the data scale is randomly set to be as large as possible.
  \item \textbf{Performance lines.}
  The numerical performance  of NED will be shown as follows:  the blue curve and its label "NED-FE"  represent the explicit forward Euler  method used in the NED method; the green curve and its label   "NED-RK2" represent  the 2nd order Runge-Kutta method used for solving NED flow; the red curve and its label   "SGD" mean the stochastic gradient descent method.
\end{itemize}

\subsection{Numerical Results of Evolution Equations for Supervised Learning}
\subsubsection{A One-Dimensional Example}

We apply the NED method to solve the supervised learning problem \eqref{intro:problemODE} to approximate $y=sin(x)$ on $[0,D]$. 

\textbf{Case 1:  $D=2\pi$.} We train 200 epochs  to compare our NED   with SGD. The initial learning rate is $\tau_0=1.0e-03$ and the min-batch size is set to be 200. The gradient flow  \eqref{odeflowUnderPara} for NED and the gradient flow  \eqref{odeflow2} for SGD are used in the numerical comparison.   We can see from Figure \ref{Fig:sinx} (a) and (c)  that 1) the NED methods with RK2 and FE solvers achieve  better accuracy than the SGD method; 2)  the NED with RK2 solver is more accurate than the NED with FE solver.  

\textbf{Case 2:  $D=10\pi$.} 
 We employ the activation function  $\sigma=a ReLU(x)+b sin(x)$ inspired by \cite{liang2021reproducing}, where  $a$ and $b$ are trainable parameters. The initial learning rate is   $\tau_0=1.0e-3$ and the epoch size is 500.  It can be seen from Figure \ref{Fig:sinx} (b) and (d)  that    the NED methods with both FE and RK2 solvers are more accurate than the SGD method while the NED with RK2 solver is more accurate than the NED with FE solver.

	\begin{figure}[!htbp]
		\centering
\subfigure[]{\includegraphics[width=.4\linewidth]{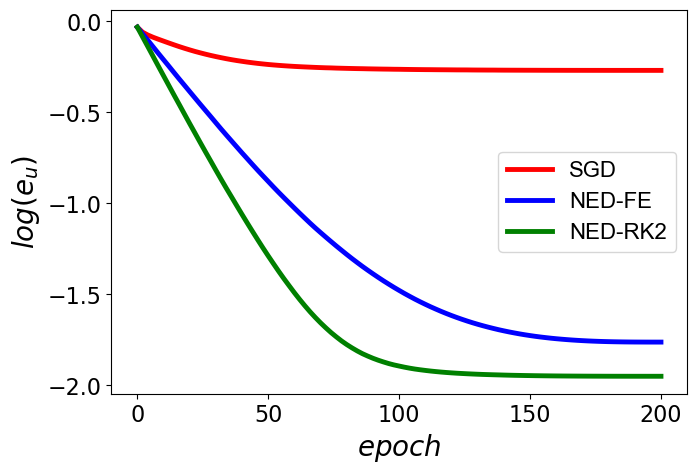}}
\subfigure[]{\includegraphics[width=.4\linewidth]{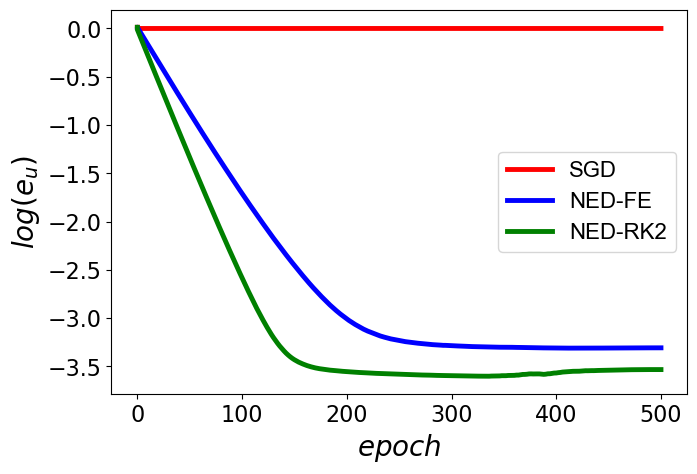}}
\subfigure[]{\includegraphics[width=.4\linewidth]{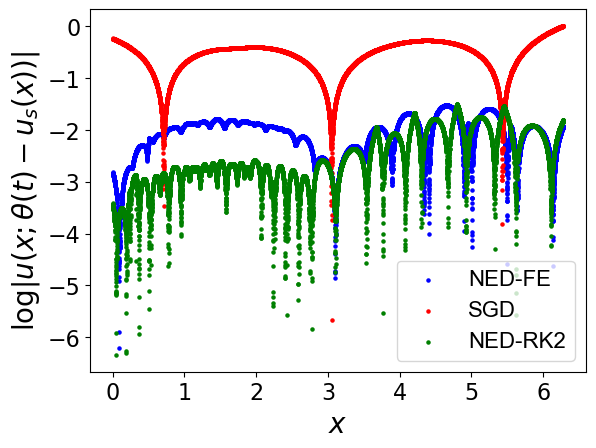}}
\subfigure[]{\includegraphics[width=.4\linewidth]{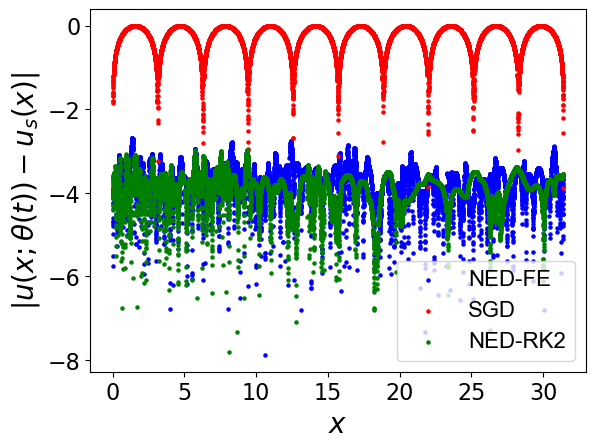}}

		\caption{Performance comparison of NED and SGD to learn $y=sin(x)$ on $[0,D]$.  (a) and (c): $D=2\pi$;  (b) and (d):  $D =10\pi$.}
		\label{Fig:sinx}
	\end{figure}

\subsubsection{A Multi-Dimensional Example}
We test the NED method     to approximate $y=\|\bm x\|_2^2$ on the domain  $[-1,1]^d$ for $d=2$, $d=10$, and $d=30$, where $\bm x=[x_1, \cdots, x_d]$.
We apply an FNN with 1 hidden layer and a width of 50. The initial learning rate $\tau_{0}$ for the NED method is set as $3.0e-03$ with 2500 epochs. The initial learning rate  for the SGD method is set as $1.0e-02$.  Each epoch of the training data has 2,000 random sample points on $[-1,1]^d$.  We can observe from  Figure \ref{Fig:x2} that for  $d=2$, $d=10$, and $d=30$, the NED methods with different ODE solvers (FE and RK2) are much more accurate than the SGD method.    We observe from Figure \ref{Fig:x2} that the green curve  for NED-RK2 seems to overlap  the blue curve for NED-FE for $d=2$ and $d=10$.   However, it should be pointed out that  the NED with RK2 solver is a bit more accurate  than the NED with  FE solver for  $d=2$, $d=10$, and $d=30$.

	\begin{figure}[!ht]
		\centering
        \includegraphics[width=0.32\linewidth]{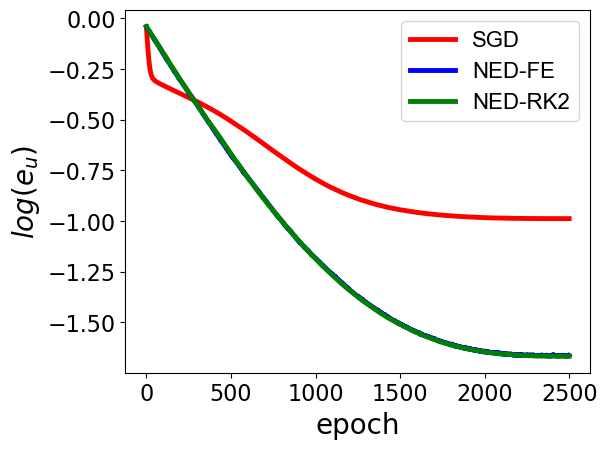}
        \includegraphics[width=0.31\linewidth]{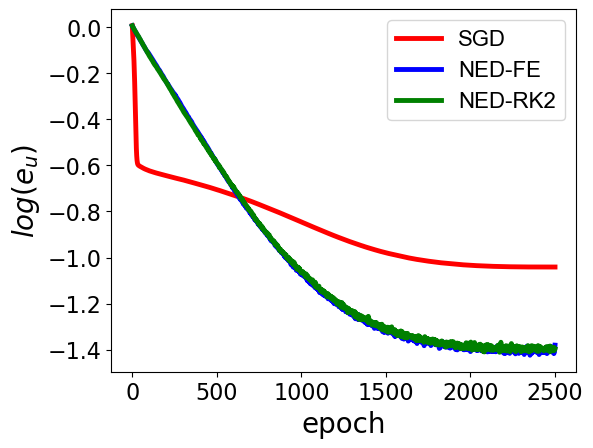}
        \includegraphics[width=0.31\linewidth]{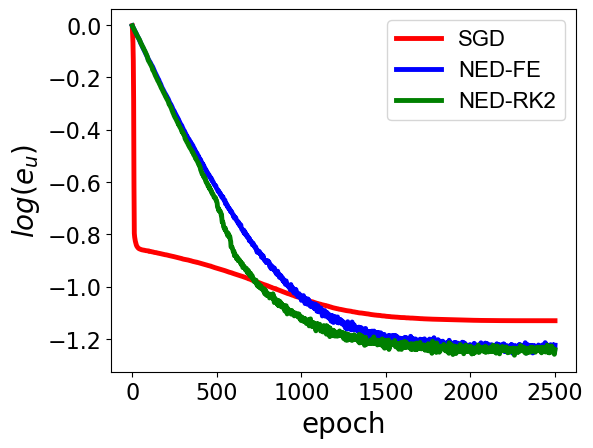}

		\caption{ Performance comparison of  NED and SGD  to approximate $y=\|\bm x\|_2^2$ on $[-1,1]^d$.  Left: $d=2$;  Middle: $d=10$; Right:   $d=30$.} 
		\label{Fig:x2}
	\end{figure}

\subsection{Numerical Results of Evolution Equations for Solving PDEs}
\subsubsection{A One-Dimensional Nonlinear Boundary Value Problem}
 We consider the following nonlinear boundary value problem that seeks $u$ such that 
\begin{equation}\label{1d}
    \begin{split}
        &\partial_{t} u-\Delta u=-2u^3, \text{in} \ [-1,0],\\
    & u(-1)=\frac{1}{2}, \qquad u(0)=\frac{1}{3}.  
    \end{split}
\end{equation} 
It is obvious that the analytical steady state solution is $u_{s}(x)=\frac{1}{x+3}$. 
Inspired by \cite{gu2020structure}, we design a special network  satisfying the boundary conditions automatically; i.e., 
$$u(x;\theta) = (x+1)(0-x)\hat{u}(x;\theta)+l_{1}(x),$$
where $l_{1}(x) = \frac{1}{2}+(\frac{1}{3}-\frac{1}{2})(x+1)$
and $\hat{u}(x;\theta)$ is the ResNet with two blocks of width 20. 
In each training iteration, we randomly select 10000 sample points in the interior domain $(-1,0)$. The initial learning rate $\tau_{0}$ for the NED method is $3.0e-04$ with 3000 epochs. The initial learning rate  for the SGD method is set as $5.0e-03$.  
We can see from Figure \ref{Fig:1DBVP} that our NED  methods with  ODE solvers (FE and RK2) trained by Algorithm \ref{alg1}  are more accurate than the SGD method when epoch $> 1500$. 


	\begin{figure}[!ht]
		\centering
\includegraphics[width=.4\linewidth]{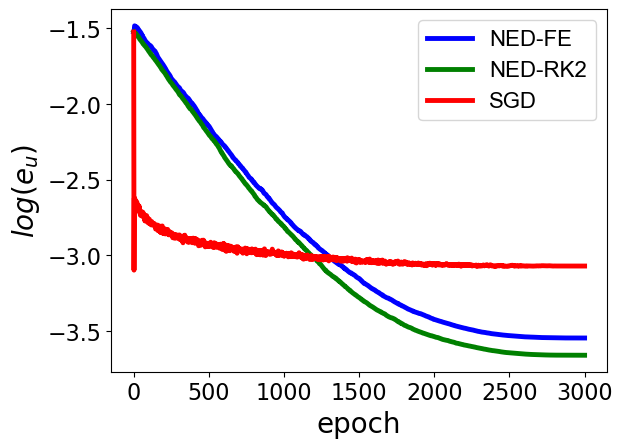}
\includegraphics[width=.39\linewidth]{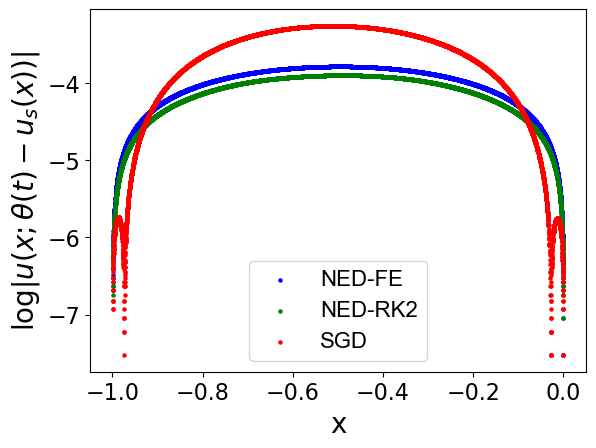}

\caption{Performance Comparison of NED and SGD.  }
		\label{Fig:1DBVP}
	\end{figure}
\subsubsection{A multi-dimension Linear Boundary Value Problem }
We consider the heat equation that seeks $u$ such that
 \begin{equation}
	\label{pde_nd}
	\left\{
		     \begin{array}{lr}
		    u_t - \Delta u = -d, & \text{in } \Omega=[0,1]^d, \\
		     u(\bm{x},t) = \frac{1}{2}\|\bm{x}\|_2^2, & \text{on } \partial \Omega.\\
		     \end{array}
	\right.
	\end{equation}
Obviously, the analytical steady state solution is  $u_{s}(\bm{x}) = \frac{1}{2}\|\bm{x}\|_2^2$. We apply a special network structure  automatically satisfying the boundary conditions introduced in \cite{gu2020structure}; i.e.,
\begin{equation}\label{construct_multi}
u(\bm{x};\theta) = \prod_{i=1}^{d}(1-x_{i})x_{i}\hat{u}(\bm{x};\theta)+\ell_{1}(\bm{x}),
\end{equation}
where $\hat{u}(\bm{x};\theta)$ is  a neural network of three hidden layers and width 20, and
\[
\ell_1(\bm{x} ) = \frac{1}{2}\|\bm{x}\|_2^2+\sin(2\pi\sum_{i=1}^{d} x_i){\displaystyle\prod_{i=1}^d x_i(1-x_i)}.
\] 

We test the case of  $d= 5$. Algorithm \ref{alg1} is coupled with the forward Euler and the 2nd Runge-Kutta ODE solvers for the semi-discretization system. We randomly choose 10000 points in the interior  of the domain. 
   The initial learning rates  for our NED method and the SGD method are 
   $\tau_{0} = 8.0e-3$ and $\tau_{0} =1.0e-1$  respectively. 
  It can be seen from  Figure \ref{Fig:MultiBVP}  that 1) the NED methods with FE and RK2 solvers are more accurate than the SGD method in updating gradient flow to gain steady state solution $u_{s}(\bm{x};\infty)$ in the evolution equation; 2) our NED method with RK2 solver is a bit more accurate than our NED method with FE solver.  

	\begin{figure}[!ht]
		\centering

\includegraphics[width=.4\linewidth]{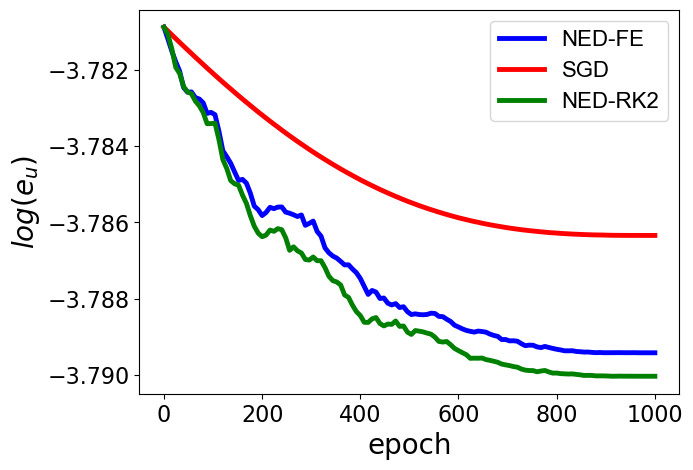}

		\caption{Performance Comparison between NED and SGD:   
    $d=5$. }
		\label{Fig:MultiBVP}
	\end{figure}

\subsubsection{A multi-dimension Nonlinear Boundary Value Problem}
We consider the following nonlinear   boundary value problem that seeks $u$ such that
\begin{equation}
	\label{pde_2d nonlinear}
	\left\{
		     \begin{array}{lr}
		     u_{t}- \Delta u+u^{3}-u= 0, & \text{in } \Omega=[0,1]^{d}, \\
             u(\bm{x},t)  = 1,  & \text{on}\  \partial \Omega.\\
		     \end{array}
	\right.
	\end{equation}
    The analytical steady-state solution is
$u_{s}(\bm{x}) = 1$.We apply a special network structure  automatically satisfying the boundary conditions introduced in \cite{gu2020structure}; i.e.,
\begin{equation}
u(\bm{x};\theta) = \prod_{i=1}^{d}(1-x_{i})x_{i}\hat{u}(\bm{x};\theta)+\ell_{1}(\bm{x}),
\end{equation}
where $\hat{u}(\bm{x};\theta)$ is  a neural network of three hidden layers and width 20, and
\[
\ell_1(\bm{x} ) = 1+\sin(2\pi\sum_{i=1}^{d} x_i){\displaystyle\prod_{i=1}^d x_i(1-x_i)}.
\] 

We test the case  for $d=5$. 
We randomly choose 20000 points in the interior domain and employ Algorithm \ref{alg1}   coupled with the forward Euler and 2nd Runge-Kutta ODE solvers to solve the semi-discretization  system. 
The initial learning rates for our NED method and the SGD method  are 
$\tau_{0}=5.0e-7$
and $\tau_{0}=5.0e-1$,
respectively. 
Figure \ref{Fig:2Dnonlinear} shows 1) the numerical performance of our NED methods with FE and RK2 solvers is  better than the SGD method;  2) our NED method  with  RK2 solver is more accurate than the NED method with FE solver.
\begin{figure}[!htbp]
		\centering
\includegraphics[width=0.4\linewidth]{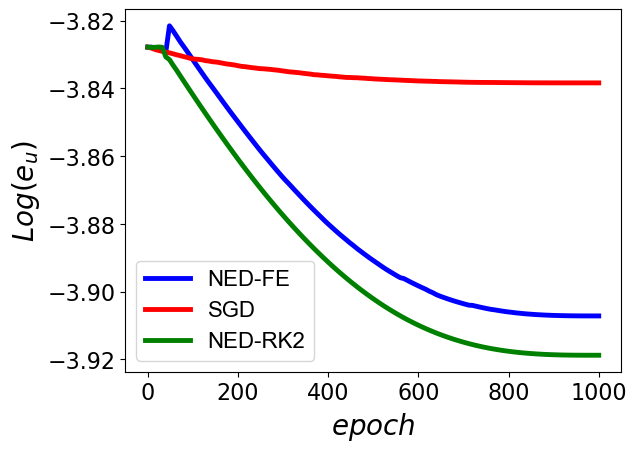}	
	\caption{Performance Comparison between NED and SGD:    
    $d=5$.}
		\label{Fig:2Dnonlinear}
	\end{figure}



\section{Conclusion}\label{sec:con}
In this paper, we propose a novel network-based optimization method called the Neural Energy Descent method (NED) to solve deep learning problems via  identifying steady-state solutions of  evolution equations. NED is powerful to be applied to a wide range of machine learning problems such as supervised learning and solving PDEs. We have developed a deep network approximation theory to justify the NED scheme. Numerical results have been observed to demonstrate the significant advantage of our NED method over the SGD method. In future work, we shall study the optimization convergence from the viewpoint of variational functional minimization and the generalization error of NED since random samples are used in the solver. This work will provide a complete theory of this new optimization algorithm for deep learning.

 \section*{Acknowledgements}
W. H. was supported by  the National Science Foundation award DMS-2052685 and the National Institutes of Health award  1R35GM146894. C. W. was partially supported by National Science Foundation under awards DMS-2136380 and DMS-2206332. H. Y. was partially supported by the US National Science Foundation under awards DMS-2244988, DMS-2206333, and the Office of Naval Research Award N00014-23-1-2007.  

\bibliographystyle{abbrv} 
\bibliography{refs} 

\newpage
\appendix

\section{Notations and Definitions}

We first introduce notations and definitions throughout this paper.

\subsection{Deep Neural Networks}

Let us summarize all basic notations used in deep neural networks as follows.
\begin{itemize}

     \item Matrices are denoted by bold uppercase letters. For instance,  $\bm{A}\in\mathbb{R}^{m\times n}$ is a real matrix of size $m\times n$, and $\bm{A}^T$ denotes the transpose of $\bm{A}$. 
     
     \item Vectors are denoted as bold lowercase letters. For example, $\bm{v}\in \R^n$ is a column vector of size $n$. Correspondingly, $\bm{v}(i)$ is the $i$-th element of $\bm{v}$. $\bm{v}=[v_1,\cdots,v_n]^T=\left[\hspace{-4pt}\begin{array}{c}
    \vspace{-3pt} v_1 \\ \vspace{-5pt} \vdots \\ v_n
     \end{array}\hspace{-4pt}\right]$ is a vector   with $\bm{v}(i)=v_i$.
     
      \item 
    A $d$-dimensional multi-index is a $d$-tuple
    $\xal=[\alpha_1,\alpha_2,\cdots,\alpha_d]^T\in \N^d.$
    Several related notations are listed below.
    \begin{itemize}
        \item  $|\xal|=|\alpha_1|+|\alpha_2|+\cdots+|\alpha_d|$;
        \item $\xx^\xal=x_1^{\alpha_1}  x_2^{\alpha_2} \cdots x_d^{\alpha_d}$, where $\xx=[x_1,x_2,\cdots,x_d]^T$;
        \item $\xal!=\alpha_1!\alpha_2!\cdots \alpha_d!$.
    \end{itemize}

    \item Let $B_{r,|\cdot|}(\xx)\subseteq \R^d$ be the closed ball with a center $\xx\subseteq \R^d$ and a radius $r$ measured by the Euclidean distance. Similarly, $B_{r,\|\cdot\|_{\ell^\infty}}(\xx)\subseteq \R^d$ is a ball measured by the discrete $\ell^\infty$-norm of a vector.
     
     
     
     \item Assume $\bm{n}\in \N^n$, then $f(\bm{n})=\mathcal{O}(g(\bm{n}))$ means that there exists positive $C$ independent of $\bm{n}$, $f$, and $g$ such that $ f(\bm{n})\le Cg(\bm{n})$ when all entries of $\bm{n}$ go to $+\infty$.
     
     \item We will use $\sigma$ to denote activation functions. Let $\sigma_1:\R\to \R$ denote the rectified linear unit (ReLU), i.e. $\sigma_1(x)=\max\{0,x\}$. With the abuse of notations, we define $\sigma_1:\R^d\to \R^d$ as $\sigma_1(\xx)=\left[\begin{array}{c}
          \max\{0,x_1\}  \\
          \vdots \\
          \max\{0,x_d\}
     \end{array}\right]$ for any $\xx=[x_1,\cdots,x_d]^T\in \R^d$.
     Furthermore, let $\sigma_2:\R\to \R$ be $\sigma_1^2$ and similarly we define the action of $\sigma_2$ on a vector $\xx$.
     
     \item We will use $\tn{NN}$ as a neural network for short and \textbf{$\sigma_r$-$\tn{NN}$ to specify an $\tn{NN}$ with activation functions $\sigma_t$ with $t\leq r$}. We will also use Python-type notations to specify a class of $\tn{NN}$s, e.g., $\sigma_1$-$\tn{NN}(\tn{c}_1;\ \tn{c}_2;\ \cdots;\ \tn{c}_m)$ is a set of ReLU FNNs satisfying $m$ conditions given by $\{\tn{c}_i\}_{1\leq i\leq m}$, each of which may specify the number of inputs ($\NNinput$), the total number of nodes in all hidden layers ($\#$node), the number of hidden layers ($\#$layer), the number of total parameters ($\#$parameter), and the width in each hidden layer (widthvec), the maximum width of all hidden layers (maxwidth), etc. For example, if $\phi\in \sigma_1$-$\tn{NN}(\NNinput=2; \NNwidth=[100,100])$,  then $\phi$  satisfies
     \begin{itemize}
         \item $\phi$ maps from $\R^2$ to $\R$.
         \item $\phi$ has two hidden layers and the number of nodes in each hidden layer is $100$.
     \end{itemize}
     \item $[n]^L$ is short for $[n,n,\cdots,n]\in \N^L$. 
     For example, \[\tn{NN}(\NNinput=d;\NNwidth=[100,100])=\tn{NN}(\NNinput=d;\NNwidth=[100]^2).\]
     
     \item For $\phi\in \tn{NN}(\NNinput=d;\NNwidth=[N_1,N_2,\cdots,N_L])$, if we define $N_0=d$ and $N_{L+1}=1$, then the architecture of $\phi$ can be briefly described as follows:
    \begin{equation*}
    \begin{aligned}
    \bm{x}=\tilde{\bm{h}}_0 \myto^{\bm{W}_1,\ \bm{b}_1} \bm{h}_1\mathop{\longrightarrow}^{\sigma} \tilde{\bm{h}}_1 \cdots \myto^{\bm{W}_L,\ \bm{b}_L} \bm{h}_L\mathop{\longrightarrow}^{\sigma} \tilde{\bm{h}}_L \mathop{\myto}^{\bm{W}_{L+1},\ \bm{b}_{L+1}} \phi(\bm{x})=\bm{h}_{L+1},
    \end{aligned}
    \end{equation*}
    where $\bm{W}_i\in \R^{N_{i}\times N_{i-1}}$ and $\bm{b}_i\in \R^{N_i}$ are the weight matrix and the bias vector in the $i$-th linear transform in $\phi$, respectively, i.e., \[\bm{h}_i :=\bm{W}_i \tilde{\bm{h}}_{i-1} + \bm{b}_i,\quad \tn{for $i=1$, $\dots$, $L+1$,}\]  and
    \[
       \tilde{\bm{h}}_i=\sigma(\bm{h}_i),\quad \tn{for $i=1$, $\dots$, $L$.}
    \]
    $L$ in this paper is also called the number of hidden layers in the literature.

     \item The expression, an FNN with width $N$ and depth $L$, means
     \begin{itemize}
         \item The maximum width of this FNN for all hidden layers less than or equal to $N$.
         \item The number of hidden layers of this FNN less than or equal to $L$.
     \end{itemize}
\end{itemize}

\begin{lemma}\label{lem:NNex1}
A list of examples and basic lemmas of $\sigma_1$-$\tn{NN}$s.
\begin{enumerate}[label=(\roman*)]
\item Any one-dimensional continuous piecewise linear function with $N$ breakpoints can be exactly realized by a one-hidden layer $\sigma_1$-$\tn{NN}$ with $N$ neurons in the hidden layer.
\item Any identity map in $\R^d$ can be carried out precisely by a $\sigma_1$-$\tn{NN}$ with one hidden layer and $2d$ neurons.
\item (Lemma 5.1 of \cite{Shen3}) For any $N,L\in \N^+$, there exists a $\sigma_1$-$\tn{NN}$ $\phi$ with width $3N$ and depth $L$ such that
    \begin{equation*}
    |\phi(x)-x^2|\le N^{-L}, \quad  \forall x\in [0,1]. 
    \end{equation*}
\item (Lemma 4.2 of \cite{Shen3}) For any $N,L\in \N^+$ and $a,b\in \R$ with $a<b$, there exists a $\sigma_1$-$\tn{NN}$ $\phi$ with width $9N+1$ and depth $L$ such that
    \begin{equation*}
    |\phi(x,y)-xy|\le 6(b-a)^2N^{-L}, \quad  \forall x,y\in [a,b].
    \end{equation*}
    \item (Proposition 4.1 of \cite{Shen3}) Assume $P(\xx)=\xx^\xal=x_1^{\alpha_1}x_2^{\alpha_2}\cdots x_d^{\alpha_d}$ for $\xal\in \N^d$ with $|\xal|=k\ge 2$. For any $N,L\in \N^+$, there exists a $\sigma_1$-$\tn{NN}$ $\phi$ with width $9(N+1)+k-2$ and depth $7k(k-1)L$ such that
    \begin{equation*}
    |\phi(\xx)-P(\xx)|\le 9(k-1)(N+1)^{-7kL},\quad  \forall \xx\in [0,1]^d.
    \end{equation*}
    \item Assume $P(\xx)= \sum_{j=1}^J c_j \xx^{\xal_j}$ for $\xal_j\in \N^d$ with $k=\max_{j}|\xal_j|\ge 2$. For any $N,L\in \N^+$, there exists a $\sigma_1$-$\tn{NN}$ $\phi$ with width $(9(N+1)+k-2)J$ and depth $7k(k-1)L$ such that
    \begin{equation*}
    |\phi(\xx)-P(\xx)|\le 9J(k-1)(N+1)^{-14L},\quad \forall \xx\in [0,1]^d.
    \end{equation*}
    \item Assume $P(\xx)=\min\{x_1,x_2,\dots,x_n\}$ for $\xx\in\R^n$, there exists a $\sigma_1$-$\tn{NN}$ $\phi(\xx)$ of width $2n$ and depth $n-1$ such that $\phi(\xx)=P(\xx)$ for any $\xx\in\R^n$.
\end{enumerate}
\end{lemma}

\begin{proof}
(i) and (ii) are simple. (iii) to (v) are quoted from \cite{Shen3}. 

Part (vi): In the case of $\min_{j}|\xal_j|\ge 2$, the $\sigma_1$-$\tn{NN}$ in (vi) can be constructed by stacking $J$ $\sigma_1$-$\tn{NN}$s approximating $\xx^{\xal_j}$ by (vi). In the case when $|\xal_j| \leq 1$ for some $j$'s, these terms can be easily taken care of using the identity map in (ii).

Part (vii): The proof of (vii) is based on the observation that $\min\{x,y\}=\frac{x+y-|x-y|}{2}$, which can be represented exactly with a $\sigma_1$-$\tn{NN}$ of width $4$ and depth $1$. We can repeatedly apply this observation to build the desired network that evaluates a $\min$ operator of two numbers per hidden layer. Each $\min$ operator takes $4$ neurons per layer and extra $2(n-2)$ neurons per layer are required to generate an identify map of dimension at least $n-2$ to pass unused numbers in $\xx$ to the next layer. Hence, the  total width requirement is $4+2(n-2)=2n$ and the depth requirement is $n-1$.
\end{proof}

\begin{lemma}\label{lem:NNex2}
A list of examples and basic lemmas of $\sigma_2$-$\tn{NN}$s.
\begin{enumerate}[label=(\roman*)]
\item $\sigma_1$-$\tn{NN}$s are $\sigma_2$-$\tn{NN}$s.
\item Any identity map in $\R^d$ can be carried out precisely by a $\sigma_2$-$\tn{NN}$ with one hidden layer and $2d$ neurons.
\item $f(x)=x^2$ can be implemented via a one-hidden-layer $\sigma_2$-$\tn{NN}$ with two neurons.
\item $f(x,y)=xy=\frac{(x+y)^2-(x-y)^2}{4}$ can be implemented via a one-hidden-layer $\sigma_2$-$\tn{NN}$ with four neurons.
\item Assume $P(\xx)=\xx^\xal=x_1^{\alpha_1}x_2^{\alpha_2}\cdots x_d^{\alpha_d}$ for $\xal\in \N^d$. For any $N,L\in \N^+$ such that $NL+2^{\lfloor \log_2 N \rfloor}\geq |\xal|$, there exists a $\sigma_2$-$\tn{NN}$ $\phi$ with width $4N+2d$ and depth $L+\lceil \log_2 N\rceil$ such that
    \begin{equation*}
    \phi(\xx)=P(\xx)\quad \tn{for any $\xx\in \R^d$.}
    \end{equation*} 
    \item Assume $P(\xx)= \sum_{j=1}^J c_j \xx^{\xal_j}$ for $\xal_j\in \N^d$. For any $N,L,a,b\in \N^+$ such that $ab\geq J$ and $(L-2b-b\log_2 N)N\geq  b \max_j |\xal_j|$, there exists a $\sigma_2$-$\tn{NN}$ $\phi$ with width $4Na+2d+2$ and depth $L$ such that
    \begin{equation*}
 \phi(\xx)=P(\xx)\quad \tn{for any $\xx\in \R^d$.}
    \end{equation*}
\end{enumerate}
\end{lemma}

\begin{proof}
(i) to (iv) are trivial. We will only prove (v) and (vi). 

Part (v): In the case of $|\xal|=k\leq 1$, the proof is simple and left for the reader. When $|\xal|=k\ge 2$, the main idea of the proof of (v) can be summarized in Figure \ref{fig:NN1}. We apply $\sigma_1$-$\tn{NN}$s to implement a $d$-dimensional identity map as in Lemma \ref{lem:NNex1} (iii). These identity maps maintain necessary entries of $\xx$ to be multiplied together. We apply $\sigma_2$-$\tn{NN}$s to implement the multiplication function in Lemma \ref{lem:NNex2} (iii) and carry out the multiplication $N$ times per layer. After $L$ layers, there are $k-NL\leq N$ multiplication to be implemented. Finally, these at most $N$ multiplications can be carried out with a small $\sigma_2$-$\tn{NN}$s in a dyadic tree structure.

Part (vi): The main idea of the proof is to apply Part (v) $J$ times to construct $J$ $\sigma_2$-$\tn{NN}$s, $\{\phi_j(\xx)\}_{j=1}^J$, to represent $\xx^{\xal_j}$ and arrange these $\sigma_2$-$\tn{NN}$s as sub-NN blocks to form a larger $\sigma_2$-$\tn{NN}$ $\tilde{\phi}(\xx)$ with $ab$ blocks as shown in Figure \ref{fig:NN2}, where each red rectangle represents one $\sigma_2$-$\tn{NN}$ $\phi_j(\xx)$ and each blue rectangle represents one $\sigma_1$-$\tn{NN}$ of width $2$ as an identity map of $\R$. There are $ab$ red blocks with $a$ rows and $b$ columns. When $ab\geq J$, these sub-NN blocks can carry out all monomials $\xx^{\xal_j}$. In each column, the results of the multiplications of $\xx^{\xal_j}$ are added up to the input of the narrow $\sigma_1$-$\tn{NN}$, which can carry the sum over to the next column. After the calculation of $b$ columns, $J$ additions of the monomials $\xx^{\xal_j}$ have been implemented, resulting in the output $P(\xx)$. 

By Part (v), for any $N\in \N^+$, there exists a $\sigma_2$-$\tn{NN}$ $\phi_j(\xx)$ of width $2d+4N$ and depth $L_j = \lceil \frac{|\xal_j|}{N}\rceil +\lceil \log_2 N \rceil$ to implement $\xx^{\xal_j}$. Note that $b\max_j L_j\leq b\left(  \frac{\max_j |\xal_j|}{N} +2 + \log_2 N \right)$. Hence, there exists a $\sigma_2$-$\tn{NN}$ $\tilde{\phi}(\xx)$ of width $2da+4Na+2$ and depth $b\left(  \frac{\max_j |\xal_j|}{N} + 2 + \log_2 N \right)$ to implement $P(\xx)$ as in Figure \ref{fig:NN2}. Note that the total width of each column of blocks is $2ad+4Na+2$ but in fact this width can be reduced to $2d+4Na+2$, since the red blocks in each column can share the same identity map of $\R^d$ (the blue part of Figure \ref{fig:NN1}).

Note that $b\left(  \frac{\max_j |\xal_j|}{N} + 2 + \log_2 N \right)\leq L$ is equivalent to $(L-2b-b\log_2 N)N\geq b \max_j |\xal_j|$. Hence, for any $N,L,a,b\in \N^+$ such that $ab\geq J$ and $(L-2b-b\log_2 N)N\geq b \max_j |\xal_j|$, there exists a $\sigma_2$-$\tn{NN}$ $\phi(\xx)$ with width $4Na+2d+2$ and depth $L$ such that $\tilde{\phi}(\xx)$ is a sub-NN of $\phi(\xx)$ in the sense of $\phi(\xx)=\tn{Id}\circ\tilde{\phi}(\xx)$ with $\tn{Id}$ as an identify map of $\R$, which means that $\phi(\xx)=\tilde{\phi}(\xx)=P(\xx)$. The proof of Part (vi) is completed.
\end{proof}

	\begin{figure}[!ht]
		\centering
		\includegraphics[width=0.8\linewidth]{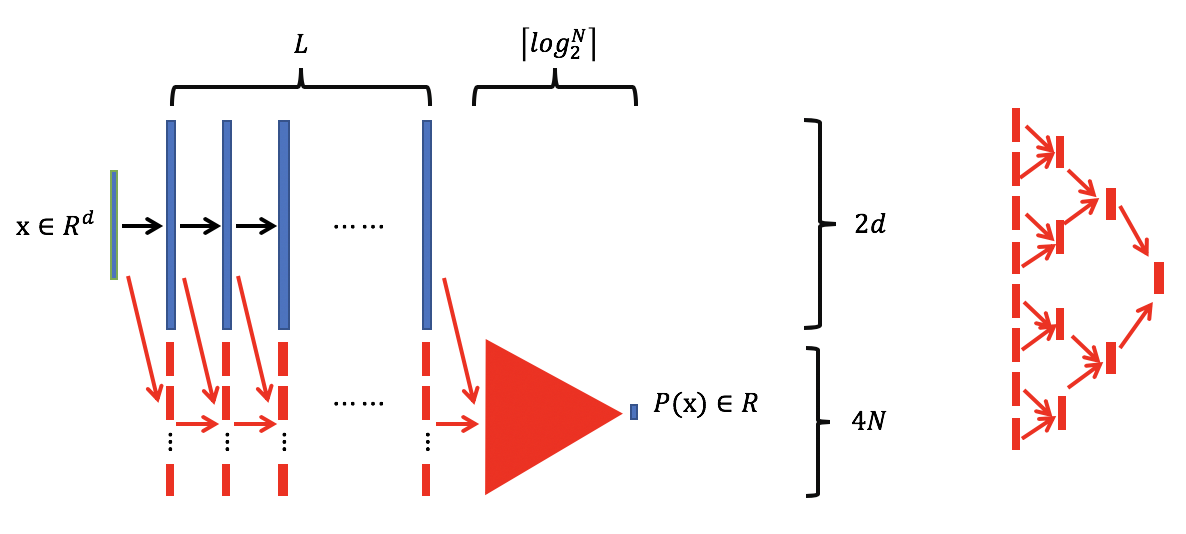}
		\caption{Left: An illustration of the proof of Lemma \ref{lem:NNex2} (v). Green vectors represent the input and output of the $\sigma_2$-$\tn{NN}$ carrying out $P(\xx)$. Blue vectors represent the $\sigma_1$-$\tn{NN}$ that implements a $d$-dimensional identity map in Lemma \ref{lem:NNex1} (iii), which was repeatedly appled for $L$ times. Black arrows represent the data flow for carrying out the identity maps. Red vectors represent the $\sigma_2$-$\tn{NN}$s implementing the multiplication function in Lemma \ref{lem:NNex2} (iii) and there $NL$ such red vectors. Red arrows represent the data flow for carrying out the multiplications. Finally, a red triangle represent a $\sigma_2$-$\tn{NN}$ of width at most $4N$ and depth at most $\lceil \log_2^N \rceil$ carrying out the rest of the multiplications. Right: An example of the red triangle is given on the right when it consists of $15$ red vectors carrying out $15$ multiplications.}
		\label{fig:NN1}
	\end{figure}

	\begin{figure}[!ht]
		\centering
		\includegraphics[width=0.8\linewidth]{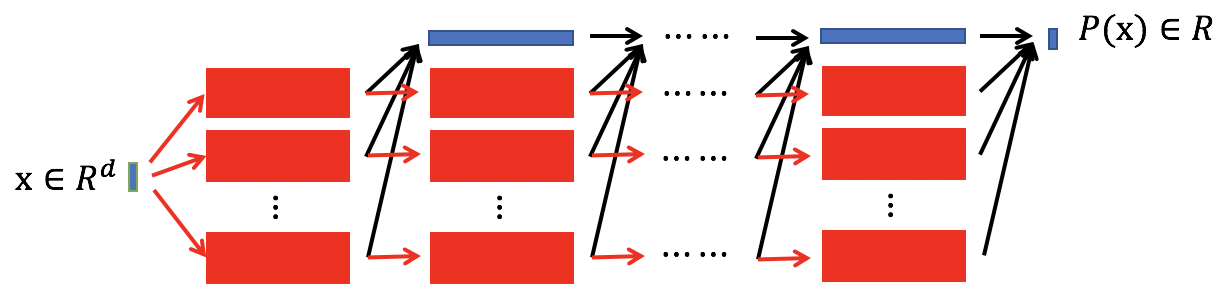}
		\caption{An illustration of the proof of Lemma \ref{lem:NNex2} (vi). Green vectors represent the input and output of the $\sigma_2$-$\tn{NN}$ $\tilde{\phi}(\xx)$ carrying out $P(\xx)$. Each red rectangle represents one $\sigma_2$-$\tn{NN}$ $\phi_j(\xx)$ and each blue rectangle represents one $\sigma_1$-$\tn{NN}$ of width $2$ as an identity map of $\R$. There are $ab\geq J$ red blocks with $a$ rows and $b$ columns. When $ab\geq J$, these sub-NN blocks can carry out all monomials $\xx^{\xal_j}$. In each column, the results of the multiplications of $\xx^{\xal_j}$ are added up to (indicated by black arrows) the input of the narrow $\sigma_1$-$\tn{NN}$, which can carry the sum over to the next column. Each red arrow passes $\xx$ to the next red block. After the calculation of $b$ columns, $J$ additions of the monomials $\xx^{\xal_j}$ have been implemented, resulting in the output $P(\xx)$.}
		\label{fig:NN2}
	\end{figure}
	
We would like to remark that it is interesting to further optimize the above lemmas so that we can optimize the approximation theories in this paper. This is left as future work.

\subsection{Sobolev spaces}

We will use $D$ to denote the weak derivative of a single variable function and $D^\xal$ to denote the partial derivative $D^{\alpha_1}_1 D^{\alpha_2}_2 \dots D^{\alpha_d}_d$ of a $d$-dimensional function with $\alpha_i$ as the order of derivative $D_i$ in the $i$-th variable and $\xal=[\alpha_1,\dots,\alpha_d]^T$. Let $\Omega$ denote an open subset of $\mathbb{R}^d$ and $L^p(\Omega)$ be the standard Lebesgue space on $\Omega$ for $p\in[1,\infty]$. We write $\nabla f:=[D_1 f, \dots, D_d f]^T$. $\partial \Omega$ is the boundary of $\Omega$. Let $\mu(\cdot)$ be the Lebesgue measure.  For $f(x)\in\Wnp(\Omega)$, we use the notation
\[
\|f\|_{\Wnp(\Omega)} = \|f\|_{\Wnp} = \|f(x)\|_{\Wnp(\Omega,\mu)},
\]
if the domain is clear from the context and we use the Lebesgue measure. If $\pdf$ is a probability density function supported in $\Omega$ with $\mu_\pdf(\cdot)$ as its corresponding measure, then we use $\|f\|_{\Wnp(\Omega,\mu_\pdf)}$ to specify the measure in the Sobolev norm.

\begin{definition} (Sobolev Space) Let $n\in \mathbb{N}_0$ and $1\leq p\leq \infty$. Then we define the Sobolev space 
\[
\Wnp(\Omega):=\{ f\in L^p(\Omega):D^\xal f \in L^p(\Omega)\text{ for all }\xal\in \mathbb{N}_0^d \text{ with }|\xal|\leq n \}
\]
with a norm
\[
\|f\|_{\Wnp(\Omega)}:=\left( \sum_{0\leq |\xal| \leq n} \|D^\xal f\|^p_{L^p(\Omega)} \right)^{1/p}
\]
and 
\[
\|f\|_{\Wni(\Omega)} := \max_{0\leq |\xal|\leq n} \| D^\xal f\|_{L^\infty(\Omega)}.
\]
\end{definition}

Many results of function approximation rely on the domain $\Omega$ and we will use the following condition on $\Omega$ (see \cite{evans1998partial}, Appendix C.1).

\begin{definition}
\label{def:Ld}
(Lipschitz-domain) We say that a bounded and open set $\Omega\subset \mathbb{R}^d$ is a Lipschitz-domain if for each $\xx_0\in\partial \Omega$ there exists $r>0$ and a Lipschitz continuous function $g:\mathbb{R}^{d-1}\rightarrow \mathbb{R}$ such that 
\[
\Omega\cap B_{r,|\cdot|}(\xx_0)=\{ \xx\in B_{r,|\cdot|}(\xx):x_d>g(x_1,\dots,x_{d-1})\},
\]
after possibly relabeling and reordering the coordinate axes, where $B_{r,|\cdot|}(\xx)$ is a sphere centered at $\xx$ with a radius $r$.
\end{definition}

In this paper, we focus on an open, bounded, and convex domain $\Omega=(0,1)^d$, which is a Lipschitz domain (see \cite{grisvard2011elliptic}, Corollary 1.2.2.3).

Let us introduce some basic lemmas of Sobolev spaces here. 

%

\begin{lemma}\label{lem:pd}
Let $f$ be any $\sigma_1$-$\tn{NN}$ and $g\in\Wnp(\Omega)$ with $1\leq p\leq \infty$, then $fg\in\Wnp(\Omega)$ and there exists a constant $C=C(d,n,p)>0$ such that
\[
|fg|_{\Wnp(\Omega)} \leq C \left(  |f|_{\WW^{1,\infty}(\Omega)} |g|_{\WW^{n-1,p}(\Omega) } + \|f\|_{L^\infty(\Omega)} |g|_{\Wnp(\Omega)}   \right).
\]
For $p=\infty$, we have $C=1$.
\end{lemma}

\begin{proof}
In the following proof, we will drop the dependence of the bounded domain $\Omega$ in all norms. Note that any $\sigma_1$-$\tn{NN}$ is a piecewise linear function and hence $D^\xal f=0$ if $\alpha_i\geq 2$ for some $i\in\{1,\dots,d\}$. Hence, $f\in\WW^{n,\infty}$ for any $n\in\N$. Note that $D^\xal (fg)=\sum_{\xal_1+\xal_2=\xal} c_{\xal_1,\xal_2}D^{\xal_1} g D^{\xal_2} f$ with constant coefficients $\{c_{\xal_1,\xal_2}\}$. By the triangle inequality and the inequalities
\[
\left( \sum_{i=1}^n a_i^p\right)^{1/p}\leq \sum_{i=1}^n a_i \leq c_1(n,p) \left( \sum_{i=1}^n a_i^p\right)^{1/p}
\]
for $a_i\geq 0$ and some $c_1(n,p)>0$, it is also easy to check that there exist a constant $c_2(d,n,p)$ such that
\[
|fg|_{\WW^{n,p}}\leq \|fg\|_{\Wnp}\leq c_2(d,n,p)\|f\|_{\WW^{n,\infty}}\|g\|_{\Wnp}.
\]
Hence, $fg\in\WW^{n,p}$. In the above inequality, $\|f\|_{\WW^{n,\infty}}$ can be replaced with $\|f\|_{\WW^{d,\infty}}$ if $n\geq d$ because $f$ is a $\sigma_1$-$\tn{NN}$. 

By the same inequalities, we can verify that 
\begin{eqnarray*}
|fg|_{\WW^{n,p}(\Omega)}&=& \left(\sum_{|\xal|=n}\| \sum_{\xal_1+\xal_2=\xal} c_{\xal_1,\xal_2} D^{\xal_1} g D^{\xal_2} f\|_{L^p}^p\right)^{1/p}  \\
 &\leq & \sum_{|\xal|=n} \| \sum_{\xal_1+\xal_2=\xal} c_{\xal_1,\xal_2} D^{\xal_1} g D^{\xal_2} f\|_{L^p} \\
 &\leq & \sum_{|\xal|=n} \sum_{\xal_1+\xal_2=\xal} |c_{\xal_1,\xal_2}|  \| D^{\xal_1} g D^{\xal_2} f\|_{L^p} \\
 &\leq & c_3(d,n)  \sum_{|\xal_1+\xal_2|=n} \| D^{\xal_1} g D^{\xal_2} f\|_{L^p} \\
 &\leq & c_3(d,n)  \sum_{|\xal_1+\xal_2|=n,\|\xal_2\|_{\ell^\infty}=1} \| D^{\xal_1} g D^{\xal_2} f\|_{L^p} \\
 & & + c_3(d,n)  \sum_{|\xal_1|=n} \| D^{\xal_1} g  f\|_{L^p} \\
 &\leq & C(d,n,p) \left(  |f|_{\WW^{1,\infty}(\Omega)} |g|_{\WW^{n-1,p}(\Omega) } + \|f\|_{L^\infty(\Omega)} |g|_{\WW^{n,p}(\Omega)}   \right).
\end{eqnarray*}

The case of $p=\infty$ is simple.
\end{proof}

\section{Deep Network Approximation in the Sobolev Space}
We prove the basic theories of Deep Network Approximation for parametrized functions in the Sobolev Space $\Wnp$ for $\sigma_1$-$\tn{NN}$s and $\sigma_2$-$\tn{NN}$s. The proofs of the theories developed here mainly follow previous works in \cite{Shen1,Shen2,Shen3,Ingo,Opschoor2019}.  The focus of \cite{Shen1,Shen2,Shen3} is on the optimal approximation rate of deep networks in terms of width $N$ and depth $L$ for continuous and $C^m$ functions, while \cite{Ingo,Opschoor2019} describe the approximation rate in terms of the number of parameters $W$ in deep networks for $\Wnp$ (with $n\in[0,1]$ and $p\in [1,\infty]$) and $\WW^{n,\infty}$ (for any $n\in \N$), respectively. We aim at characterizing deep network approximation for solution manifolds of parametric PDEs depending on $d_t$-dimensional parameters $t\in \R^{d_t}$. We focus on the case when the solutions $u(\xx,t)$ of PDEs are in the space $\Wnp\times C^m$. The theories developed here are new and are motivated by the NED method, where we consider the semi-discretization of $u(\xx,t)$ via a neural network $U(\xx;\xth(t))$, in which the smoothness of $\xth(t)$ is crucial for the justification of the semi-discretization scheme. Our theories can also be applied in other applications when the semi-discretization is applied, e.g., Uncertanty Quantification (UQ), where a mathematical model is described by a PDE parametrized by $t$. It is well know that PDE models in UQ analytically depend on $t$ (see \cite{CHKIFA2015400} for example) and hence the smoothness of $\xth(t)$ in $U(\xx;\xth(t))$ is required for the application of the semi-discretization.

\subsection{Preliminaries for Averaged Taylor Polynomials}

We provide several well-known lemmas and deep network approximation results mainly following \cite{Ingo} and \cite{Shen3}, where the key observation that polynomials and local Taylor expansions can be efficiently approximated by deep neural networks is proposed in \cite{yarotsky2017}.

We first introduce the averaged Taylor expansion for parametrized functions generalized from the averaged Taylor expansion in \cite{Ingo}.

\begin{definition}
\label{def:atp}
(Averaged Taylor Polynomial) Let $m,n\in\N$, $1\leq p\leq \infty$, and $f(\xx,t)\in\WW^{n-1,p}(\Omega)\times C^m(\Omega_t)$. Let $\xx_0\in\Omega$, $r>0$ such that for the ball $B:=B_{r,|\cdot|}(\xx_0)$ it holds that $B\subset \Omega$. The corresponding Taylor polynomial of order $n$ of $f(\xx,t)$ averaged over $B$ is defined for $x\in\Omega$ and each $t\in\Omega_t$ as
\begin{equation}
\label{eqn:atp1}
Q^n f(\xx,t):=\int_B T^n_{\yy} f(\xx,t) \phi(\yy) d\yy,
\end{equation}
where
\begin{equation}
\label{eqn:atp2}
T^n_{\yy} f(\xx,t) :=\sum_{|\xal|\leq n-1} \frac{1}{\xal !} D^{\xal} f(\yy,t) (\xx-\yy)^{\xal},
\end{equation}
and $\phi$ is an arbitrary cut-off function supported in $\overline{B}$, i.e.
\[
\phi\in C_c^\infty (\R^d)\text{ with }\phi(\xx)\geq 0\text{ for all }x\in\R^d, \tn{ supp }\phi=\overline{B}\text{, and }\int_{\R^n} \phi(\xx)d\xx=1.
\]
\end{definition}

Definition \ref{def:atp} is a generalization of Definition B.7 in \cite{Ingo} from $f\in\WW^{n-1,p}$ to $\WW^{n-1,p}\times C^m$. Following the proof of Lemma B.9 in \cite{Ingo}, we can show the below lemma, the proof of which is an immediate result of the Lebesgue's dominated convergence theorem and the fact that the average Taylor polynomial is a finite sum of monomials with coefficients linearly depending on a finite terms of derivatives of $f$, and the derivatives of $f$ is in the $L^p(B)$ with $p\in[1,\infty]$ and $B$ bounded. Hence, the proof is left for the reader.

\begin{lemma}
\label{lem:atp}
Let $m,n\in\N$, $1\leq p\leq \infty$, and $f(\xx,t)\in\WW^{n-1,p}(\Omega)\times C^m(\Omega_t)$. Let $\xx_0\in\Omega$, $r>0$, and $R\geq 1$ such that for the ball $B:=B_{r,|\cdot|}(\xx_0)$ it holds that $B\subset \Omega$ and $B\subset B_{R,\|\cdot\|_{\ell^\infty}}(0)$. The corresponding Taylor polynomial of order $n$ of $f(\xx,t)$ averaged over $B$ can be written as
\begin{equation}\label{eqn:para}
Q^n f(\xx,t)=\sum_{|\xal|\leq n-1} c_\xal(t) \xx^\xal
\end{equation}
for $\xx\in\Omega$ and $t\in \Omega_t$. Moreover, there exists a function $c=c(n,d,R)>0$ such that the coefficient functions $c_\xal(t)$ are in $C^m(\Omega_t)$ and bounded with $|c_\xal(t)|\leq c r^{-d/p} \|f(\xx,t)\|_{\WW^{n-1,p}(\Omega)}$ for all $\xal$ with $|\xal|\leq n-1$.
\end{lemma}

As discussed in \cite{Ingo}, unlike the standard Taylor expansion that derives a truncated approximation around a point $\xx_0$, the average Taylor expansion relies on a ball $B$ and requires that the path between each $\xx_0\in B$ and each $\xx\in\Omega$ is contained in $\Omega$. This geometrical condition can be better interpreted if we introduce the following definitions before we apply the average Taylor polynomials in deep network approximation.

\begin{definition}
\label{def:star}
(Star-Shaped) Let $\Omega$, $B\subset \R^d$. Then $\Omega$ is called star-shaped with respect to $B$ if
\[
\overline{\text{conv}}(\{x\}\cup B)\subset \Omega \qquad \text{ for  all }x\in \Omega.
\]
\end{definition}

Next, the chunkiness of a domain $\Omega$ introduced below is important in the family of subdivisions of $\Omega$ for averaged Taylor expansions.

\begin{definition}
\label{def:chun}
(Chunkiness) Let $\Omega\subset \R^d$ be bounded. We define the set
\[
\mathcal{R}(\Omega):=\left\{ r>0: \text{there exists } \xx_0\in\Omega\text{ such that }\Omega\text{ is star-shaped w.r.t. } B_{r,|\cdot|}(\xx_0) \right\}.
\]
If $\mathcal{R}(\Omega)\neq \emptyset$, then we define $r^*_{\text{max}}(\Omega):=\sup \mathcal{R}(\Omega)$ and call $\gamma(\Omega):=\frac{\text{diam}(\Omega)}{r^*_{\text{max}}(\Omega)}$ the chunkiness parameter of $\Omega$.
\end{definition}

Recall that performing the averaged Taylor expansion locally is the key idea of deep network approximation for smooth functions. Hence, we introduce the partition of unity in \cite{yarotsky2017,Ingo} below for the purpose of a self-contained analysis with a slight modification.

\begin{lemma}
\label{lem:pu}
For any $d,K\in\N$ there exists a collection of functions
\[
\Psi=\{ \phi_\xk:\xk\in\{0,1,\dots,K\}^d\}
\]
with $\phi_\xk:\R^d\rightarrow \R$ for all $\xk\in\{0,\dots,K\}^d$ with the following properties:
\begin{enumerate}[label=(\roman*)]
\item $0\leq \phi_\xk(\xx)\leq 1$ for every $\phi_\xk\in\Psi$ and every $\xx\in \R^d$;
\item $\sum_{\phi_\xk\in\Psi} \phi_\xk(\xx)=1$ for every $\xx\in[0,1]^d$;
\item $\tn{supp }\phi_\xk\subset B_{1/K,\|\cdot\|_{\ell^\infty}}(\xk/K)$ for every $\phi_\xk\in\Psi$;
\item there exists a constant $c\geq 1$ such that $\|\phi_\xk\|_{L^\infty(\R^d)}\leq 1$ and $\|\phi_\xk\|_{\WW^{n,\infty}(\R^d)}\leq c \cdot K$ for $n\geq 1$;
\item there exists an absolute constant $c_1\geq 1$ such that for each $\phi_\xk\in\Psi$ there is a $\sigma_1$-$\NN$ $\Phi_\xk$ with a $d$-dimensional input, a $d$-dimensional output, one hidden layer, at most $6d$ neurons per layer, that satisfies
\[
\phi_\xk=\prod_{\ell=1}^d (\Phi_\xk)_\ell,
\]
$\|(\Phi_\xk)_\ell\|_{L^\infty}\leq 1$, and $\|(\Phi_\xk)_\ell\|_{\WW^{n,\infty}}\leq c_1 K$ for $n\geq 1$ and for all $\ell=1,\dots,d$, where $(\Phi_k)_\ell$ is the $\ell$-th output of $\Phi_\xk$.


\item there exists an absolute constant $c_2\geq 1$ such that for each $\phi_\xk\in\Psi$ there exists a $\sigma_2$-$\NN$ $\Phi_\xk$ with a $d$-dimensional input, a one-dimensional output, at most $\lceil \log_2(d)\rceil+1$ hidden layers and at most $\max\{4,2d\}$ neurons per hidden layer, that satisfies
\[
\phi_\xk=\Phi_\xk,
\]
$\|\Phi_\xk\|_{\WW^{n,\infty}}\leq (c_2 K)^n$ for all $n\in\{0,1,\dots,d\}$, and $\|\Phi_\xk\|_{\WW^{n,\infty}}\leq (c_2 K)^d$ for all $n\geq d+1$.
\end{enumerate}
\end{lemma}

\begin{proof}
The proof of Part (i) to (iii) can be found in Lemma C.3 in \cite{Ingo}, which also proves Part (iv) when $n=0$ and $1$. The $\phi_\xk$ introduced in Lemma C.3 in \cite{Ingo} has an explicit formula as follows:
\begin{equation}
\label{eqn:pt}
\phi_\xk(\xx):=\prod_{\ell=1}^d \psi\left( 3K\left( x_\ell-\frac{k_\ell}{K}\right)\right),
\end{equation}
where
\[
\psi:\R\rightarrow\R,\qquad \psi(x):=  
\begin{cases}
1 &\text{for $|x|<1$},\\
0 &\text{for $2<|x|$},\\
2-|x| &\text{for $1\leq|x|\leq 2$}.
\end{cases}
\]
The case of $n>1$ in Part (iv) is also true since $\phi_\xk$ is a piecewise linear function. 

Part (v) in this paper is generalized from Part (v) of Lemma C.3 in \cite{Ingo} as well and it is also true since $\phi_\xk$ is a piecewise linear function. 

In Part (vi), the construction of the $\sigma_2$-$\tn{NN}$ is based on the fact that: 1) A one-hidden layer $\sigma_2$-$\tn{NN}$ with width $2$ can exactly carry out the square function; 2) A one-hidden layer $\sigma_2$-$\tn{NN}$ with width $4$ can exactly implement a multiplication function by Lemma \ref{lem:NNex2}. Hence, the target $\sigma_2$-$\tn{NN}$ consists of two parts: the first part is the $\sigma_1$-$\tn{NN}$ in Part (v); the second part only has $\sigma_2$ activation functions with width $2d$ carrying out $\frac{d}{2^s}$ multiplications in the $s$-th hidden layer. The bounds of $\Phi_\xk$ in the Sobolev norm is given by the fact that $\Phi_\xk=\phi_\xk$ with an explicit formula in \eqref{eqn:pt}.
\end{proof}

Finally, we introduce a lemma to quantify the approximation error of the  averaged Taylor approximation in the Sobolev semi-norm.

\begin{lemma}
\label{lem:BH}
(Bramble-Hilbert) Let $\Omega\subset \mathbb{R}^d$ be open and bounded, $x_0\in\Omega$ and $r>0$ such that $\Omega$ is star-shaped with respect to $B:=B_{r,|\cdot|}(x_0)$, and $r>(1/2)r^*_{max}(\Omega)$. Moreover, let $n\in\mathbb{N}$, $1\leq p\leq \infty$ and denote by $\gamma(\Omega)$ the chunkiness parameter of $\Omega$. Then there exists a constant $C=C(n,d,\gamma)>0$ such that for all $f\in \WW^{n,p}(\Omega)\times C^m(\Omega_t)$
\[
|f(\xx,t)-Q^n f (\xx,t)|_{\Wkp(\Omega)} \leq C h^{n-k} |f(\xx,t)|_{\Wnp(\Omega)}\quad \text{ for } k=0,1,\dots,n,
\]
where $Q^n f(\xx,t)$ denotes the Taylor polynomial of order $n$ of $f(\xx,t)$ for a fixed $t$ averaged over $B$ and $h=\tn{diam}(\Omega)$.
\end{lemma}
The proof of Lemma \ref{lem:BH} for a fixed $t$ can be found in Lemma 4.3.8 of \cite{Brenner2008}. The proof of Lemma \ref{lem:BH} can be easily obtained and we leave it for the reader.

Now we are ready to quantify localized polynomial approximation in the Sobolev space using the partition of unity in Lemma \ref{lem:pu} and the Bramble-Hilbert Lemma as follows.

%

\begin{lemma}
\label{lem:atpa}
Let $d,K\in\N$, $n\in\N_+$, $s\in\N$ with $s\leq n-1$, $1\leq p\leq \infty$, and $\Psi=\Psi(d,K)=\{\phi_{\xk}:\xk\in\{0,\dots,K\}^d\}$ be the partition of unity from Lemma \ref{lem:pu}. Then for any $f(\xx,t)\in \Wnp((0,1)^d)\times C^m(\Omega_t)$, there exist polynomials $\pkt(\xx)=\sum_{\xal\in\N_0^d,|\xal|\leq n-s} c_{f,\xk,\xal,s}(t) \xx^\xal$ for $\xk\in\{0,\dots,K\}^d$ with the following properties:

\begin{enumerate}[label=(\roman*)]
\item Let $f_K:=\sum_{\xk\in\{0,\dots,K\}^d} \phi_\xk \pkt$, then the operator $T_s:\Wnp((0,1)^d)\times C^m(\Omega_t)\rightarrow\WW^{s,p}((0,1)^d)\times C^m(\Omega_t)$ with $T_s f=f-f_K$ is linear and bounded with
\[
\|T_sf\|_{\WW^{s,p}((0,1)^d} \leq C_s \left( \frac{1}{K} \right)^{n-s} \|f\|_{\Wnp((0,1)^d)}
\]
for some constant $C_s=C_s(n,d,p)$. 
\item Furthermore, there is a function $c_s=c_s(d,n)>0$ such that the coefficients of the polynomials $\pkt$ satisfy
\[
|c_{f,\xk,\xal,s}(t)|\leq c_s\|\tilde{f}\|_{\Wnp(\Omega_{\xk,K})} K^{d/p}
\]
for all $\xal$ with $|\xal|\leq n-s$ and $\xk\in\{0,\dots,K\}^d$, where $\Omega_{\xk,K}:=B_{\frac{1}{K},\|\cdot\|_{\ell^\infty}}\left(\frac{\xk}{K}\right)$ and $\tilde{f}\in\Wnp(\R^d)\times C^m(\Omega_t)$ is an extension of $f$.
\end{enumerate}
\end{lemma}

\begin{proof}
For a fixed $t$, the proof of this lemma is similar to the proof of Lemma C.4 in \cite{Ingo} and the first part of Theorem 1 in \cite{yarotsky2017}. The idea is to use the partition of unity and the averaged Taylor polynomials to derive local approximations. The global approximation is the combination of local approximations and its error can be estimated via the Bramble-Hilbert Lemma \ref{lem:BH}.

Let $E:\Wnp((0,1)^d)\times C^m(\Omega_t) \rightarrow \Wnp(\R^d)\times C^m(\Omega_t)$ be the extension operator of the domain $\Omega=(0,1)^d$ from \cite{10.2307/j.ctt1bpmb07} (Theorem VI.3.1.5) and set $\tf:=Ef$. Note that 
\begin{equation}\label{eqn:gb3}
|\tf|_{\WW^{s,p}(\Omega)}\leq \|\tf\|_{\Wnp(\R^d)}\leq C_E \|f\|_{\Wnp((0,1)^d)},
\end{equation}
for arbitrary $\Omega\subset (0,1)^d\subset \R^d$ and $1\leq s\leq n$, where $C_E=C_E(d,n,p)$ is the norm of the extension operator. 

\textbf{Step 1(Averaged Taylor polynomials):} For each $\xk\in\{0,\dots,K\}^d$, we set
\[
\OkK := B_{\frac{1}{N},\|\cdot\|_{\ell^\infty}}\left( \frac{\xk}{K}\right) \quad\text{and}\quad \BkK := B_{\frac{3}{4K},|\cdot|}\left( \frac{\xk}{K} \right),
\]
and denote by $p_{\xk}=\pkt$ the Taylor polynomial of order $n$ of $\tf$ averaged over $\BkK$ (see Definition \ref{def:atp}). It follows from Lemma \ref{lem:atp} that we can write $p_\xk=\sum_{|\xal|\leq n-s} c_{f,\xk,\xal,s}(t) \xx^\xal$ and that there is a constant $\bar{c}_s=\bar{c}_s(n,d)>0$ such that
\[
|c_{f,\xk,\xal,s}|\leq \bar{c}_s\|\tf\|_{\Wnp(\OkK)}\left( \frac{3}{4K}\right)^{-d/p}\leq c_s \|\tf\|_{\Wnp(\OkK)} K^{d/p}
\]
for $\xk\in\{0,\dots,K\}^d$, where $c_s=c_s(d,n,p)>0$ is a suitable constant.  Hence, Part (ii) of this lemma is proved.

\textbf{Step 2 (Local estimates in $\|\cdot\|_{\Wsp}$):} It is easy to check that under the setting of this lemma, the conditions of the Bramble-Hilber Lemma \ref{lem:BH} are satisfied (See Step $2$ of the proof of Lemma C.4 of \cite{Ingo}). Hence, we can derive the accuracy of local approximations via
\begin{equation}\label{eqn:fn0}
\|\tf -\pk\|_{L^p(\OkK)}\leq \tilde{c}_1 \left( \frac{2\sqrt{d}}{K}\right)^n |\tf|_{\Wnp(\OkK)}\leq \tilde{c}_2 \left(\frac{1}{K}\right)^n \|\tf\|_{\Wnp(\OkK)}.
\end{equation}
Here, $\tilde{c}_1=\tilde{c}_1(n,d)>0$ is from Lemma \ref{lem:BH} that only depends on $n$ and $d$, since the chunkiness parameter of $\OkK$ is a constant depending only on $d$. $\tilde{c}_2=\tilde{c}_2(n,d)>0$ is chosen as a suitable constant. Similarly, we have
\begin{equation}\label{eqn:fn1_2}
|\tf-\pk|_{\WW^{s,p}(\OkK)}\leq \tilde{c}_{s+2}\left( \frac{1}{K}\right)^{n-s}\|\tf\|_{\Wnp(\OkK)},
\end{equation}
for $s=1,2,\dots,n$, where $\tilde{c}_{s+2}=\tilde{c}_{s+2}(n,d)>0$ is a suitable constant. Combining \eqref{eqn:fn0} and \eqref{eqn:fn1_2} with the cut-off functions in the partition of unity, we have
\begin{equation}\label{eqn:loc1}
\|\phi_\xk(\tf-\pk)\|_{L^p(\OkK)}\leq \|\phi_\xk\|_{L^\infty(\OkK)}\|\tf-\pk\|_{L^p(\OkK)}\leq \tilde{c}_2\left(\frac{1}{K}\right)^n \|\tf\|_{\Wnp(\OkK)}.
\end{equation}
Since the cut-off functions are $\sigma_1$-$\tn{NN}$s, using Lemma \ref{lem:pd}, we have
\begin{eqnarray}\label{eqn:loc2}
|\phi_\xk(\tf-\pk)|_{\WW^{s,p}(\OkK)}&\leq& C' |\phi_\xk|_{\WW^{1,\infty}(\OkK)}|\tf-\pk|_{\WW^{s-1,p}(\OkK)}\nonumber \\
 &&+ C' \|\phi_\xk\|_{L^\infty(\OkK)}|\tf-\pk|_{\WW^{s,p}(\OkK)}\nonumber\\
 &\leq &C'\tilde{c}K\tilde{c}_{s+1}\left(\frac{1}{K}\right)^{n-s+1}\|\tf\|_{\Wnp(\OkK)} \nonumber\\
 & &+ C' \tilde{c}_{s+2}\left(\frac{1}{K}\right)^{n-s}\|\tf\|_{\Wnp(\OkK)} \nonumber\\
 &=& \tilde{C}_s \left(\frac{1}{K}\right)^{n-s}\|\tf\|_{\Wnp(\OkK)}
\end{eqnarray}
for $s=1,2,\dots,n$, where $\tilde{c}$ is an absolute constant, $C'=C'(d,s,p)>0$ and $\tilde{C}_s = \tilde{C}_s(d,p)>0$. By \eqref{eqn:loc1} and \eqref{eqn:loc2}, it holds that
\begin{equation}\label{eqn:loc3}
\|\phi_\xk(\tf-\pk)\|_{\WW^{s,p}(\OkK)}\leq \bar{C}_s\left(\frac{1}{K}\right)^{n-s}\|\tf\|_{\Wnp(\OkK)}
\end{equation}
for some $\bar{C}_s=\bar{C}_s(d,p)>0$ for $s=1,2,\dots,n$.

\textbf{Step 3 (Global estimate in the $\|\cdot\|_{\Wsp}$):} By the property of the partition of unity, we have
\[
\tf(\xx)=\sum_{\xk\in\{0,\dots,K\}^d} \phi_\xk(\xx)\tf(\xx),\quad\text{for a.e. }x\in(0,1)^d.
\]
Note that $\tf$ is an extension of $f$. Hence, for $s\in\{0,1,\dots,n\}$, we have
\begin{eqnarray}\label{eqn:gb0}
\|f-\sum_{\xk\in\{0,\dots,K\}^d} \phi_\xk \pk \|^p_{\WW^{s,p}(\zod)} &=&
\|\tf-\sum_{\xk\in\{0,\dots,K\}^d} \phi_\xk \pk \|^p_{\WW^{s,p}(\zod)} \nonumber\\
&=& \|\sum_{\xk\in\{0,\dots,K\}^d} \phi_\xk (\tf-\pk) \|^p_{\WW^{s,p}(\zod)} \\
&\leq & \sum_{\tilde{\xk}\in\{0,\dots,K\}^d}  \|\sum_{\xk\in\{0,\dots,K\}^d} \phi_\xk (\tf-\pk) \|^p_{\WW^{s,p}(\Omega_{\tilde{\xk},K})},\nonumber
\end{eqnarray}
where in the last step we apply the same partition of unity with a different index $\tilde{\xk}$ to decompose the domain of the $\WW^{s,p}$-norm. Note that for  
\begin{eqnarray}\label{eqn:gb1}
\|\sum_{\xk\in\{0,\dots,K\}^d} \phi_\xk (\tf-\pk) \|_{\WW^{s,p}(\Omega_{\tilde{\xk},K})} &\leq &
\sum_{\substack{\xk\in\{0,\dots,K\}^d\\ \|\xk-\tilde{\xk}\|_{\ell^\infty}\leq 1}} \| \phi_\xk (\tf-\pk) \|_{\WW^{s,p}(\Omega_{\tilde{\xk},K})} \nonumber\\
&\leq &\sum_{  \substack{\xk\in\{0,\dots,K\}^d\\ \|\xk-\tilde{\xk}\|_{\ell^\infty}\leq 1}  } \| \phi_\xk (\tf-\pk) \|_{\WW^{s,p}(\Omega_{{\xk},K})} \\
&\leq &  \bar{c}_s \left( \frac{1}{K}\right)^{n-s} \sum_{ \substack{\xk\in\{0,\dots,K\}^d\\ \|\xk-\tilde{\xk}\|_{\ell^\infty}\leq 1}  } \| \tf\|_{\WW^{n,p}(\Omega_{{\xk},K})} ,\nonumber
\end{eqnarray}
where $\bar{c}_s:=\tilde{c}_2$ if $s=0$ and $\bar{c}_s:=\bar{C}_s$ if $s\in\{1,2,\dots,n\}$. In the proof of \eqref{eqn:gb1}, we have used the triangle inequality together with the support property (iii) in Lemma \ref{lem:pu} in the first inequality; the support property (iii) is used in the second inequality; and the last inequality comes from \eqref{eqn:loc1} and \eqref{eqn:loc3}.

Finally, by the definition of $f_K$, \eqref{eqn:gb0}, and \eqref{eqn:gb1}, we have
\begin{eqnarray}\label{eqn:gb4}
\|f-f_K\|_{\Wsp(\zod)}^p &\leq & \sum_{\tilde{\xk}\in\{0,\dots,K\}^d}  \bar{c}_s^p  \left( \frac{1}{K}\right)^{p(n-s)}  \left( \sum_{  \substack{\xk\in\{0,\dots,K\}^d\\ \|\xk-\tilde{\xk}\|_{\ell^\infty}\leq 1}  } \| \tf\|_{\WW^{n,p}(\Omega_{{\xk},K})}\right)^p\nonumber\\
(\text{H{\"o}lder's inequality})\quad &\leq & \bar{c}_s^p  \left( \frac{1}{K}\right)^{p(n-s)}   \sum_{\tilde{\xk}\in\{0,\dots,K\}^d}  \sum_{  \substack{\xk\in\{0,\dots,K\}^d\\ \|\xk-\tilde{\xk}\|_{\ell^\infty}\leq 1}  } \| \tf\|_{\WW^{n,p}(\Omega_{{\xk},K})}^p 3^{dp/q}\nonumber\\
&\leq &  \bar{c}_s^p  3^{dp/q} \left( \frac{1}{K}\right)^{p(n-s)}  3^d \sum_{\tilde{\xk}\in\{0,\dots,K\}^d}  \| \tf\|_{\WW^{n,p}(\Omega_{{\xk},K})}^p \nonumber\\
&\leq &  \bar{c}_s^p  3^{dp/q} \left( \frac{1}{K}\right)^{p(n-s)}  3^d 2^d   \| \tf\|_{\WW^{n,p}(\cup_{\tilde{\xk}\in\{0,\dots,K\}^d}\Omega_{{\xk},K})}^p \nonumber\\
&=&  \bar{c}_s^p  3^{dp/q} \left( \frac{1}{K}\right)^{p(n-s)}  3^d 2^d   \| \tf\|_{\WW^{n,p}(\R^d)}^p
\end{eqnarray}
where $\frac{1}{p}+\frac{1}{q}=1$ and the last two steps comes from the definition of the partition of unity. By \eqref{eqn:gb3} and \eqref{eqn:gb4}, we have
\[
\|T_sf\|_{\WW^{s,p}((0,1)^d)}= \|f-f_K\|_{\Wsp(\zod)}  \leq C_s  \left( \frac{1}{N} \right)^{n-s} \|f\|_{\Wnp((0,1)^d)} 
\]
for $s\in\{0,1,\dots,n\}$, where $C_s =C_E \bar{c}_s  3^{d/q} \left( \frac{1}{K}\right)^{n-s}  6^{d/p} $. Hence, Part (i) of this lemma is proved.
\end{proof}

\subsection{Main Theorems}

We will prove our main approximation theory theorems here. The approximation rates in these theorems are certainly not tight. Following the ideas in \cite{Shen1,Shen2,Shen3,yarotsky18a,yarotsky19,Ingo,Opschoor2019}, nearly optimal approximation rates can be derived in the Sobolev space; but network parameters cannot be continuous in the target function by the theory of optimal nonlinear approximation \cite{Ron} (Theorem 4.2), which is not desired in our NED model.

First, we will show a theorem quantifying the approximation capacity of $\sigma_1$-$\tn{NN}$s to approximate continuous functions on $[0,1]^d$ to warn up. The approximation error will be measured in the $\WW^{0,p}$-norm with $p\in[1,\infty]$ using the modulus of continuity of a function $f$ defined via
\begin{equation*}
\omega_f(r):= \sup\big\{|f(\bm{x})-f(\bm{y})|:\bm{x},\bm{y}\in [0,1]^d,\ \|\bm{x}-\bm{y}\|_2\le r\big\},\quad \tn{for any $r\ge 0$}.
\end{equation*}


\begin{proof}[Proof of Theorem \ref{thm:ap0}]
When $t$ is fixed, it was proved in \cite{yarotsky18a} that there exists a $\sigma_1$-$\tn{NN}$ $\Phi(\xx;\xth(t))$ such that 
\begin{equation}
\label{eqn:ap0sum}
\Phi(\xx;\xth(t))=\sum_{\xk\in\{0,\dots,K\}^d} \phi(K\xx-\xk) f(\xk/K,t)\approx f(\xx,t),
\end{equation}
where $\phi(\xx)$ is the ``spike function" from $\R^d\rightarrow\R$ defined as
\[
\phi(\xx):=\max\left\{0, \min\left( \min_{k\neq s}(1+x_k-x_s), \min_k(1+x_k)  , \min_k(1-x_k) \right)\right\}.
\]
By Lemma \ref{lem:NNex1} (vii), the spike function can be exactly represented as a $\sigma_1$-$\tn{NN}$ with width $12+2d$ and depth $d^2-d+1$. This $\sigma_1$-$\tn{NN}$ performs up to three $\min$-operators per layer for $d^2-d+1$ layers, each of these $\min$-operators is for the three components of $\phi(\xx)$ inside its firs $\min$. In the first component inside the first $\min$ of $\phi(\xx)$, there are $d^2-d$ $\min$-operators, resulting in the dominated depth of the network of $\phi(\xx)$. 

The approximation error of $\Phi(\xx;\xth(t))$ is 
\[
\|\Phi(\xx;\theta(t))-f(\xx,t)\|_{L^\infty([0,1]^d)}\leq 3d\cdot\omega_{f(\xx,t)}(\frac{1}{K}).
\]
The construction of $\Phi$ is visualized in Figure \ref{fig:NN6} and it is clear that only the parameters in the last linear combination of $\Phi$ depend on $f(\xx,t)$.

To prove Theorem \ref{thm:ap0} in this paper, we will arrange the $\sigma_1$-$\tn{NN}$s of $\{\phi(K\xx-\xk)\}_{\xk\in\{0,\dots,K\}^d}$ into $a$ rows and $b$ columns with $ab\geq M=(K+1)^d$ as in Figure \ref{fig:NN7}. The sum in \eqref{eqn:ap0sum} is carried out via partial linear combination indicated by black arrows in Figure \ref{fig:NN7}.

Originally, the width of each red block in Figure \ref{fig:NN7} is $12+2d$ and the depth is $d^2-d+1$. Since there is an identity map of $\xx\in\R^d$ in the blue blocks, the width of red blocks can be reduced to $12$. Hence, given $N$ and $L$ as the total width and depth of $\Phi$, the constraint $ab\geq M=(K+1)^d$ becomes
\[
\left\lfloor \frac{N-2-2d}{12} \right\rfloor  \left\lfloor \frac{L}{d^2-d+1} \right\rfloor \geq (K+1)^d.
\]
Hence, the largest possible $K$ is
\[
K=\left( \left\lfloor \frac{N-2-2d}{12} \right\rfloor  \left\lfloor \frac{L}{d^2-d+1} \right\rfloor  \right)^{1/d}-1,
\]
leading to the final error estimation
\begin{eqnarray*}
\|\Phi(\xx;\theta(t))-f(\xx,t)\|_{L^p([0,1]^d)}&\leq & \|\Phi(\xx;\theta(t))-f(\xx,t)\|_{L^\infty([0,1]^d)}\\
&\leq & 3d\cdot\omega_{f(\xx,t)}\left(\frac{1}{\left( \left\lfloor \frac{N-2-2d}{12} \right\rfloor  \left\lfloor \frac{L}{d^2-d+1} \right\rfloor  \right)^{1/d}-1}\right)
\end{eqnarray*}
for $p\in[1,\infty)$. Note that $a\geq 1$ and $b\geq 1$. Hence, we require that $N\geq 2d+14$ and $L\geq d^2-d+1$.

Finally, most parameters in $\Phi(\xx;\xth(t))$ are constants independent of $f(\xx,t)$. The parameters depending on $f(\xx,t)$ are the linear combination coefficients in the sum of \eqref{eqn:ap0sum}. Hence, since $f(\xx,t)$ is in $C^m(\Omega_t)$ for a fixed $\xx$, $\xth(t)$  is also in $C^m(\Omega_t)$.
\end{proof}

\begin{figure}[!ht]
		\centering
		\includegraphics[width=0.7\linewidth]{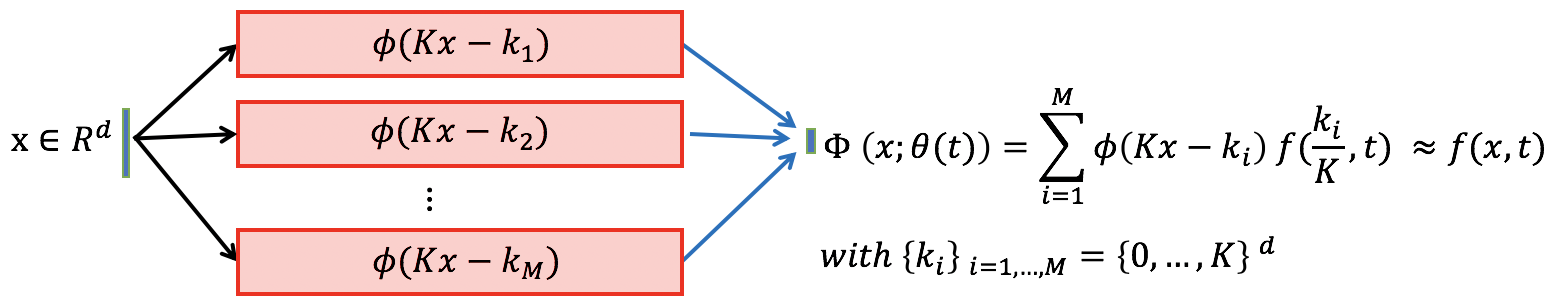}
		\caption{An illustration of the proof of Proposition $1$ in \cite{yarotsky18a}. Green vectors represent the input and output of the $\sigma_1$-$\tn{NN}$ $\Phi(\xx;\xth(t))$ carrying out $\sum_{\xk\in\{0,\dots,K\}^d} \phi(K\xx-\xk) f(\xk/K,t)\approx f(\xx,t)$. We order all vectors $\xk\in\{0,\dots,K\}^d$ as $\{\xk_1,\xk_2,\dots,\xk_M\}$ with $M=(K+1)^d$. $\Phi(\xx;\xth(t))$ consists of $M$ basic sub-$\sigma_1$-$\tn{NN}$s, each of which exactly represents $\phi(K\xx-\xk_i)$ for $i=1,\dots,M$. Blue arrows represent the computation related to $f(\xx,t)$: a linear combination of $\{\phi(K\xx-\xk_i)\}$ with coefficients as $f(\xk_i/K,t)$; other parts of the network $\Phi$ are independent of $f$. 
		 }
		\label{fig:NN6}
	\end{figure}

\begin{figure}[!ht]
		\centering
		\includegraphics[width=1\linewidth]{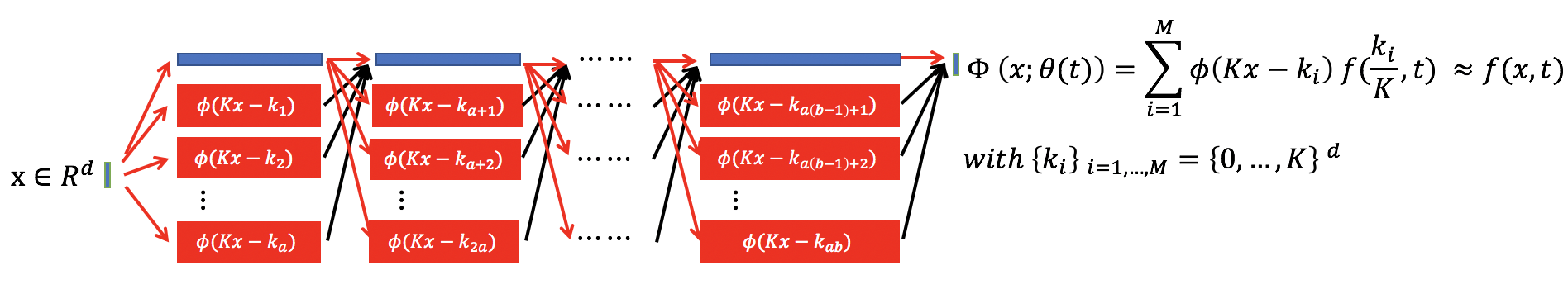}
		\caption{An illustration of the proof of Theorem \ref{thm:ap0}. Green vectors represent the input and output of the $\sigma_1$-$\tn{NN}$ $\Phi(\xx;\xth(t))$ carrying out $\sum_{\xk\in\{0,\dots,K\}^d} \phi(K\xx-\xk) f(\xk/K,t)\approx f(\xx,t)$. We order all vectors $\xk\in\{0,\dots,K\}^d$ as $\{\xk_1,\xk_2,\dots,\xk_M\}$ with $M=(K+1)^d\leq ab$. $\Phi(\xx;\xth(t))$ consists of $M$ basic sub-$\sigma_1$-$\tn{NN}$s, each of which exactly represents $\phi(K\xx-\xk_i)$ for $i=1,\dots,M$. These blocks are arranged into $a$ rows and $b$ columns. Blue blocks with width $2d+2$ consists of an identity map of $\xx\in\R^d$ and an identity map of $\R$ per layer. Red arrows pass $\xx$ to red sub-NNs in the next column or pass the partial sum  $\sum_{i=1}^{aj} \phi(K\xx-\xk_i) f(\xk_i/K)$  after the $j$-th column to the blue sub-NN in the next column. Black arrows represent the computation related to $f(\xx,t)$: a linear combination of $\{\phi(K\xx-\xk_i)\}$ with coefficients as $f(\xk_i/K,t)$; other parts of the network $\Phi$ are independent of $f$. 
		 }
		\label{fig:NN7}
	\end{figure}

\begin{figure}[!ht]
		\centering
		\includegraphics[width=1\linewidth]{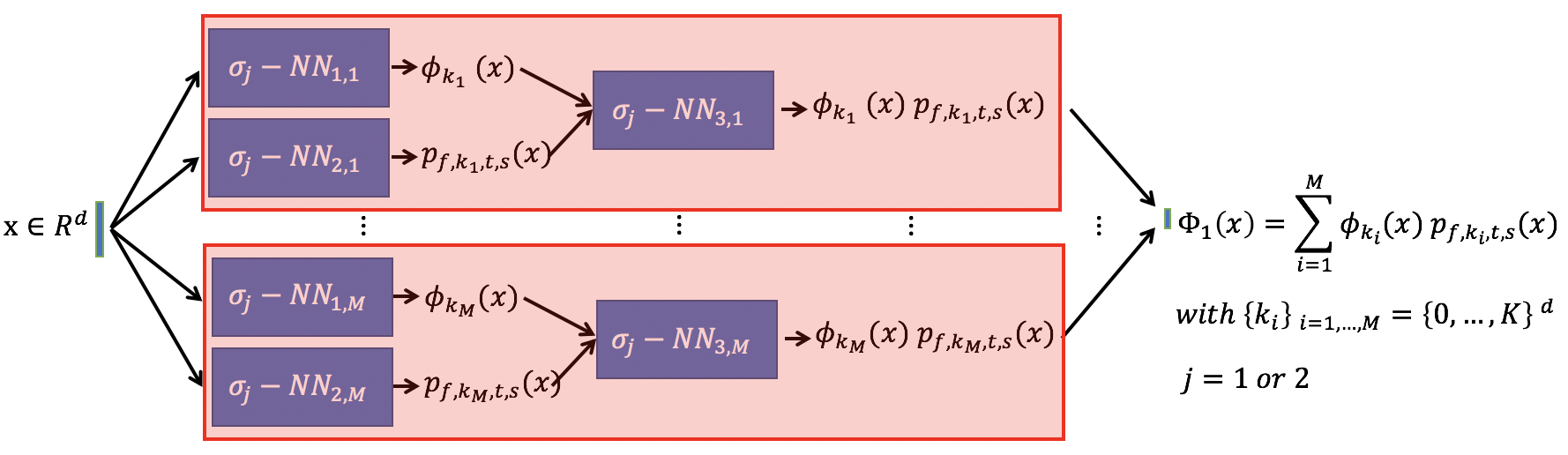}
		\caption{An illustration of 
  the proof of Part (i) in Theorem \ref{thm:ap2} when $j=2$. Green vectors represent the input and output of the $\sigma_j$-$\tn{NN}$ $\Phi_1(\xx)$ carrying out $\sum_{\xk\in\{0,\dots,K\}^d} \phi_\xk \pkt$. Black arrows represent the computational flow of either an identify map for imputing or outputting a variable, or a summation of several numbers.  We order all vectors $\xk\in\{0,\dots,K\}^d$ as $\{\xk_1,\xk_2,\dots,\xk_M\}$ with $M=(K+1)^d$. $\Phi_1(\xx)$ consists of $M$ basic sub-$\sigma_j$-$\tn{NN}$s denoted as $\sigma_j$-$\tn{NN}_{i}$ for $i=1,\dots,M$. Each $\sigma_j$-$\tn{NN}_{i}$ has three main components: $\sigma_j$-$\tn{NN}_{1,{i}}$, $\sigma_j$-$\tn{NN}_{2,{i}}$, and $\sigma_j$-$\tn{NN}_{3,{i}}$.  
		 }
		\label{fig:NN3}
	\end{figure}

	\begin{figure}[!ht]
		\centering
		\includegraphics[width=1\linewidth]{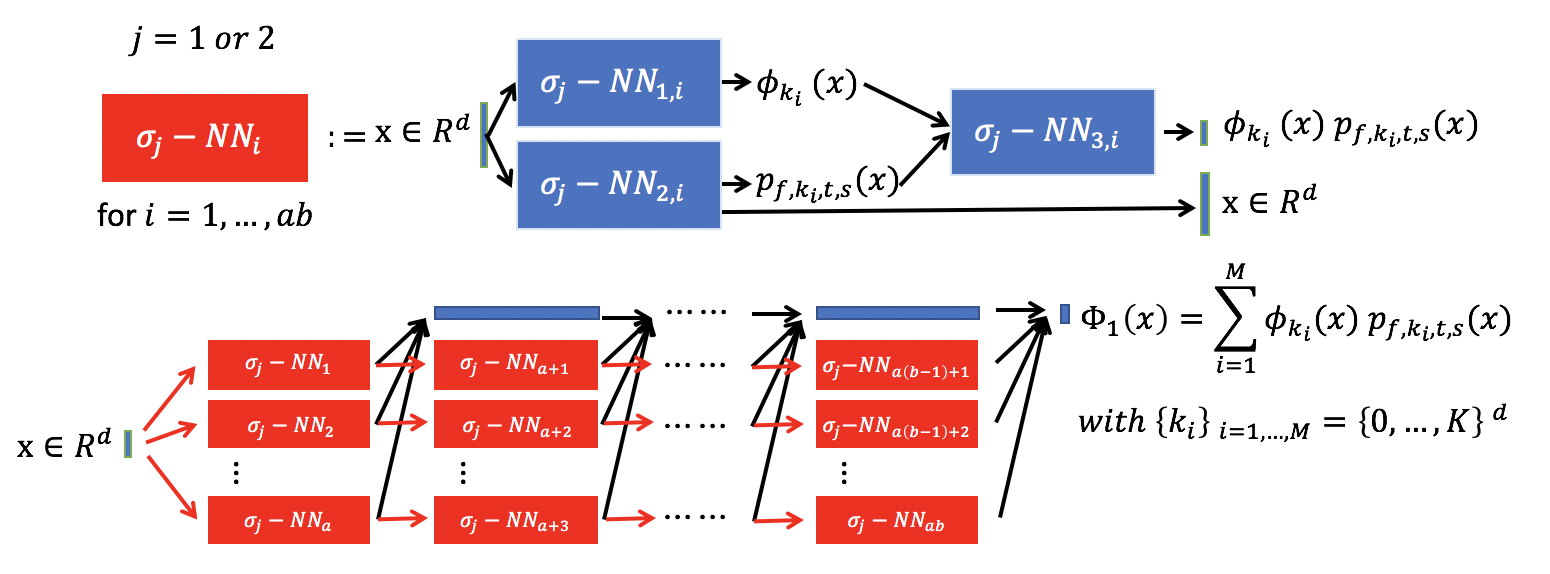}
		\caption{An illustration of the proof of Part (ii) in 
  Theorem \ref{thm:ap2} when $j=2$.  Green vectors represent the input and output of the $\sigma_j$-$\tn{NN}$ $\Phi_2(\xx)$ carrying out $\sum_{\xk\in\{0,\dots,K\}^d} \phi_\xk \pkt$. Black arrows represent the computational flow of either an identify map of $\R$ or a summation of several numbers. Red arrows represent the data flow of an identify map of $\xx$. The construction essentially arranges the basic sub-$\sigma_j$-$\tn{NN}$ blocks, each of which is denoted as $\sigma_j$-$\tn{NN}_i$ that represents $\phi_{\xk_i}p_{f,\xk_i,t,s}$ shown on top of this figure, into $a$ rows and $b$ columns such that $ab\geq M=(K+1)^d$. The narrow blue sub-$\sigma_j$-$\tn{NN}$s  in $\Phi_2$ calculate a partial sum and move it to the next column of sub-$\sigma_j$-$\tn{NN}$s.  }
		\label{fig:NN4}
	\end{figure}

Next, we present the approximation theory of $\sigma_2$-$\tn{NN}$s in the Sobolev space $\Wnp(\Omega)\times C^m(\Omega_t)$ as follows.

\begin{proof}[Proof of Theorem \ref{thm:ap2}]
In Lemma \ref{lem:atpa}, we have constructed local polynomials to approximate $f$ with an approximation error measured by the $\Wsp$-norm. We will apply Lemma \ref{lem:NNex2} to show that the sum of local polynomials can be represented by a $\sigma_2$-$\tn{NN}$ $\Phi_1$ or $\Phi_2$ satisfying the requirements in Part (i) and (ii) in this theorem.

Let $K$ be a sufficiently large integer to be determined later. For the given $d$, $K$, $n$, $p$, $s$, let $\Psi=\Psi(d,K)=\{\phi_{\xk}:\xk\in\{0,\dots,K\}^d\}$ be the partition of unity from Lemma \ref{lem:pu}. By Lemma \ref{lem:atpa}, for the given $f\in \Wnp((0,1)^d)\times C^m(\Omega_t)$, there exist polynomials $\pkt(\xx)=\sum_{|\xal|\leq n-s} c_{f,\xk,\xal,s}(t) \xx^\xal$ for $\xk\in\{0,\dots,K\}^d$ such that
\begin{equation}\label{eqn:req4}
\|f-\sum_{\xk\in\{0,\dots,K\}^d} \phi_\xk \pkt\|_{\WW^{s,p}((0,1)^d} \leq C_s  \left( \frac{1}{K} \right)^{n-s} \|f\|_{\Wnp((0,1)^d)}
\end{equation}
for some constant $C_s =C_s(n,d,p)$.

\textbf{Proof of Part (i).} If we choose $K=\left\lceil \left(\frac{C_s \|f\|_{W^{n,p}((0,1)^d)}}{\epsilon} \right)^{1/(n-s)}\right\rceil$, then by \eqref{eqn:req4}
\[
\|f-\sum_{\xk\in\{0,\dots,K\}^d} \phi_\xk \pkt\|_{\WW^{s,p}((0,1)^d} \leq \epsilon.
\] We will show that there is a $\sigma_2$-$\tn{NN}$ $\Phi_1(\xx;\xth_1(t))$ representing $\sum_{\xk\in\{0,\dots,K\}^d} \phi_\xk \pkt$ in the above equation satisfying the requirements of width, depth, and the number of parameters in Part (i). Then Part (i) is proved. 

We have visualized the construction of the $\sigma_2$-$\tn{NN}$ $\Phi_1$ in Figure \ref{fig:NN3}. We order all vectors $\xk\in\{0,\dots,K\}^d$ as $\{\xk_1,\xk_2,\dots,\xk_M\}$ with $M=(K+1)^d$. For each $\xk_i$, we construct a sub-$\sigma_2$-$\tn{NN}$ $\sigma_2$-$\tn{NN}_{i}$ consisting of three $\sigma_2$-$\tn{NN}$s: $\sigma_2$-$\tn{NN}_{1,{i}}$ representing the partition of unity function $\phi_{\xk_i}$; $\sigma_2$-$\tn{NN}_{2,{i}}$ carrying out $p_{f,\xk_i,t,s}$; and $\sigma_2$-$\tn{NN}_{3,{i}}$ for a multiplication function. Then the output of this sub-$\sigma_2$-$\tn{NN}$ is $\phi_{\xk_i}p_{f,\xk_i,t,s}$. Finally, one more layer of $\tn{NN}$ sums up $\phi_{\xk_i}p_{f,\xk_i,t,s}$ to get $\Phi_1=\sum_{i=1}^M \phi_{\xk_i} p_{f,\xk_i,t,s}$.

For each $i$, the existence of $\sigma_2$-$\tn{NN}_{1,{i}}$ is given by Lemma \ref{lem:pu} and $\sigma_2$-$\tn{NN}_{1,{i}}$ has a width at most $\max\{4,2d\}$ and a depth at most $1+\lceil \log_2 d \rceil$. For each $i$, the construction of $\sigma_2$-$\tn{NN}_{2,{i}}$ is given by Lemma \ref{lem:NNex2} (vi) and  $\sigma_2$-$\tn{NN}_{2,{i}}$ has a width at most $4Na+2d+2$ and a depth at most $L$ for any $(a,b,N,L)\in \N^4$ such that $ab\geq 2d^{n-s}$ and $(L-2b-b\log_2 N)N\geq b(n-s)$, since there are at most $2d^{n-s}$ terms in the sum and the maximum degrees of all these terms is $n-s$. By choosing $a=1$, $b=2d^{n-s}$, $N=1$, and $L=(4+2(n-s))d^{n-s}$, $\sigma_2$-$\tn{NN}_{2,{i}}$ has a width $2d+6$ and a depth $(4+2(n-s))d^{n-s}$. Note that $\sigma_2$-$\tn{NN}_{3,{i}}$ for all $i$ can be constructed by Lemma \ref{lem:NNex2} (iv) using width $4$ and depth $1$. Hence, the total width and depth of $\Phi_1$ is at most $(K+1)^d(2d+6)$ and $1+(4+2(n-s))d^{n-s}$, respectively. Note that $K=\left\lceil \left(\frac{C_s \|f\|_{W^{n,p}((0,1)^d)}}{\epsilon} \right)^{1/(n-s)}\right\rceil$. Hence, the width is at most $(2d+6)\left( \left( \frac{C_s \|f\|_{\Wnp((0,1)^d)}}{\epsilon} \right)^{1/(n-s)}+2\right)^d\leq (2d+6)\left(  \frac{C_s \|f\|_{\Wnp((0,1)^d)}}{\epsilon} +2^{n-s} \right)^{d/(n-s)}$.

Actually, each $\sigma_2$-$\tn{NN}_{2,{i}}$ has a sub-$\sigma_2$-$\tn{NN}$ of width $2d$ copying the input $\xx$ almost to its end (see the blue blocks in Figure \ref{fig:NN1} that illustrates the basic building block of $\sigma_2$-$\tn{NN}_{2,{i}}$) and a sub-$\sigma_2$-$\tn{NN}$ of width $2$ for summing up real numbers (see the blue blocks in Figure \ref{fig:NN2}). Hence, it is redundant for all $\sigma_2$-$\tn{NN}_{2,{i}}$'s to have their own sub-$\sigma_2$-$\tn{NN}$'s. If they share the same sub-$\sigma_2$-$\tn{NN}$, then the total width of $\Phi_1$ can be reduced to $2d+2+4\left(  \frac{C_s \|f\|_{\Wnp((0,1)^d)}}{\epsilon} +2^{n-s} \right)^{d/(n-s)}$.

\textbf{Proof of Part (ii).} In part $(ii)$, given the budget of the total width $ {N}$ and the total depth $ {L}$, we will identify how large $K$ can be so as to construct $\Phi_2$ within the budget. Suppose $K$ has been determined, then the construction of $\Phi_2$ is as follows.

Suppose $a$ and $b$ are two positive integers such that 
\begin{equation}\label{eqn:ab}
ab\geq M=(K+1)^d
\end{equation}
 and their values are to be determined later. We construct $\Phi_2(\xx;\xth_2(t))$ in Part (ii) using $a$ rows and $b$ columns of sub-$\sigma_2$-$\tn{NN}$s denoted as $\{\sigma_2$-$\tn{NN}_{i}\}_{i=1}^{ab}$ listed in red in Figure \ref{fig:NN4}. $\sigma_2$-$\tn{NN}_{i}$ is constructed in almost the same form as $\sigma_2$-$\tn{NN}_{i}$ in Part (i) shown in red in Figure \ref{fig:NN3}; the only difference is that $\sigma_2$-$\tn{NN}_{i}$ in $\Phi_2$ carries over its input $\xx$ as a part of its output. The capacity of passing $\xx$ to its output is due to the deign of its sub-$\sigma_2$-$\tn{NN}$ $\sigma_2$-$\tn{NN}_{i,2}$, which contains an identify map of $\R^d$ in each layer as shown in Figure \ref{fig:NN1}. Hence, $\sigma_2$-$\tn{NN}_{i}$ in $\Phi_2$ as shown in Figure \ref{fig:NN4} is able to compute $\phi_{\xk_i}p_{f,\xk_i,t,s}$ with an input $\xx$ from the previous column; the results of $\sigma_2$-$\tn{NN}_{i}$ in the same column are added up to obtain an accumulated partial sum, which is carried over to the next column of $\Phi_2$ via an identify map in blue in Figure \ref{fig:NN4}. Finally, after $b$ columns of partial sums, we obtain $\Phi_1=\sum_{i=1}^M \phi_{\xk_i} p_{f,\xk_i,t,s}$.

The size of $\sigma_2$-$\tn{NN}_{i}$ in $\Phi_2$ is the same as that of $\sigma_2$-$\tn{NN}_{i}$ in $\Phi_1$, i.e., width at most $2d+6$ and depth at most $(4+2(n-s))d^{n-s}$. Since we have arranged $a$ rows and $b$ columns of $\sigma_2$-$\tn{NN}_{i}$'s, the total width of $\Phi_2$, which is $ {N}$, should satisfy 
\begin{equation}\label{eqn:req1}
 {N}\geq 2+a(2d+6)
\end{equation}
 and the total depth, which is $ {L}$, should satisfy 
 \begin{equation}\label{eqn:req2}
  {L}\geq b(4+2(n-s))d^{n-s}. 
 \end{equation}

Actually, each $\sigma_2$-$\tn{NN}_{2,{i}}$ in $\sigma_2$-$\tn{NN}_{{i}}$ has a sub-$\sigma_2$-$\tn{NN}$ of width $2d$ copying the input $\xx$ almost to its end (see the blue blocks in Figure \ref{fig:NN1} that illustrates the basic building block of $\sigma_2$-$\tn{NN}_{2,{i}}$) and a sub-$\sigma_2$-$\tn{NN}$ of width $2$ for summing up real numbers (see the blue blocks in Figure \ref{fig:NN2}). Hence, it is redundant for all $\sigma_2$-$\tn{NN}_{2,{i}}$'s to have their own sub-$\sigma_2$-$\tn{NN}$'s. If they share the same sub-$\sigma_2$-$\tn{NN}$, we can sharpen \eqref{eqn:req1} to 
\begin{equation}\label{eqn:req3}
 {N}\geq 2+4a +2d,
\end{equation}
 where $2$ is for the shared partial sum sub-NN and $2d$ is for the shared identify map sub-NN. In sum, by \eqref{eqn:req2} and \eqref{eqn:req3}, we know the largest possible $a$ is $\lfloor \frac{ {N}-2-2d}{4}\rfloor$, and the largest possible $b$ is $\lfloor \frac{ {L}}{(4+2(n-s))d^{n-s}} \rfloor$. By the requirement in \eqref{eqn:ab}, the largest possible $K$ should satisfy
 \[
 \lfloor \frac{ {L}}{(4+2(n-s))d^{n-s}} \rfloor\lfloor \frac{ {N}-2-2d}{4}\rfloor\geq M=(K+1)^d.
 \]
Hence, the largest possible $K$ is
\begin{equation}\label{eqn:req5}
K= \left(\lfloor \frac{ {L}}{(4+2(n-s))d^{n-s}} \rfloor\lfloor \frac{ {N}-2-2d}{4}\rfloor\right)^{1/d}-1.
\end{equation}
By \eqref{eqn:req4} and \eqref{eqn:req5}, we have
\begin{eqnarray*}
&&\|\Phi_2(\xx;\xth_2(t))-f(\xx,t)\|_{\Wsp((0,1)^d)}\\
&=&
\|f-\sum_{\xk\in\{0,\dots,K\}^d} \phi_\xk \pkt\|_{\WW^{s,p}((0,1)^d}\\
&\leq & C_s  \left( \frac{1}{K} \right)^{n-s} \|f\|_{\Wnp((0,1)^d)}\\
&\leq & C_s  \left( \frac{1}{\left(\lfloor \frac{ {L}}{(4+2(n-s))d^{n-s}} \rfloor\lfloor \frac{ {N}-2-2d}{4}\rfloor\right)^{1/d}-1} \right)^{n-s} \|f\|_{\Wnp((0,1)^d)}\\
&\leq & \frac{\bar{C}_s  \|f\|_{\Wnp((0,1)^d)} }{\left( {N} {L}  \right)^{(n-s)/d}},
\end{eqnarray*}
where $\bar{C}_s =\bar{C}_s(n,d,p)$. So, we have proved Part (ii) in this Theorem.

\textbf{Proof of the smoothness of $\xth_1(t)$ and $\xth_2(t)$.} The proof of the existence of $\Phi_1(\xx;\xth_1(t))$ and $\Phi_2(\xx;\xth_2(t))$ above is constructive and hence we know their parameter sets $\xth_1(t)$ and $\xth_2(t)$ explicitly. Most of these parameters are constants independent of $f(\xx,t)$ and hence they are in $C^m(\Omega_t)$ as a constant function in $t$. The sub-NNs depending on $f(\xx,t)$ in $\Phi_1(\xx;\xth_1(t))$ and $\Phi_2(\xx;\xth_2(t))$ are the sub-NNs denoted as  $\sigma_2$-$\tn{NN}_{2,i}$ carrying out $p_{f,\xk_i,t,s}$ for $i=1,\dots,M$. Recall that the construction of each $\sigma_2$-$\tn{NN}_{2,i}$ is illustrated in Figure \ref{fig:NN2}, where the parameters depending on $f(\xx,t)$ comes from the summation represented by black arrows in Figure \ref{fig:NN2}. These parameters depending on $f$ are given by the coefficients of the polynomial of $\xx$ in \eqref{eqn:para}. By \eqref{eqn:atp1} and \eqref{eqn:atp2}, these coefficients are equal to a finite linear combination of terms of the form
\begin{equation*}
\int_B \frac{1}{\xal !} D^{\xal} \tilde{f}(\yy,t) \yy^{\bar{\xal}}  \phi_{\xk_i} (\yy) d\yy,
\end{equation*}
where $|\bar{\xal}|\leq |\xal|\leq n-s$, $\tilde{f}$ is the extension of $f$, $\phi_{\xk_i}$ is a cut-off function with $i$ corresponding to $\sigma_2$-$\tn{NN}_{2,i}$, $B$ is a bounded compact set belong to the support of $\phi_{\xk_i}$. Since $|\phi_{\xk_i} (\yy)|\leq 1$, $| \yy^{\bar{\xal}}|$ is bounded on $B$, and $f(\yy,t)$ is in $\Wnp((0,1)^d)\times C^m(\Omega_t)$ on $B$, by the Lebesgue's dominated convergence theorem, we know that the non-constant parameters in $\xth_1(t)$ and $\xth_2(t)$ are functions in $C^m(\Omega_t)$. Hence, we have completed the proof of Theorem \ref{thm:ap2}.
\end{proof}

\end{document}